\documentclass[abstract=no]{scrartcl} 
\usepackage[margin=3.8cm]{geometry}

\usepackage{amsmath}
\usepackage{bbm}
\usepackage{latexsym}
\usepackage[english]{babel}
\usepackage{amsfonts}
\usepackage[T1]{fontenc}
\usepackage[utf8]{inputenc}
\usepackage{amsthm}
\usepackage{amssymb}
\usepackage{dsfont}
\usepackage{version}
\usepackage{caption}
\usepackage[backend=bibtex,style=alphabetic,hyperref=true]{biblatex}
\usepackage[psamsfonts]{eucal}
\usepackage{stmaryrd}
\usepackage{url}
\usepackage{hyperref}
\usepackage{color}
\usepackage{blindtext}
\usepackage{tikz}
\usetikzlibrary{arrows,matrix,positioning}
\tikzstyle{dmatrix}=[matrix of math nodes,row sep=2.5em, column sep=2.5em,
text height=1.5ex, text depth=0.25ex] 
\usepackage{mathtools}
\usepackage{microtype}
\usepackage{graphicx}
\usepackage{lipsum}
\usepackage{float}
\usepackage{csquotes}
\usepackage{empheq}
\usepackage{enumitem}
\DeclareSymbolFont{epsilon}{OML}{ntxmi}{m}{it}
\DeclareMathSymbol{\epsilon}{\mathord}{epsilon}{"0F}

\theoremstyle{plain}
\newtheorem{theorem}{Theorem}[section]
\newtheorem{lemma}[theorem]{Lemma}
\newtheorem{prop}[theorem]{Proposition}
\newtheorem{cor}[theorem]{Corollary}

\newtheorem{prop/Def}[theorem]{Propsition/Definition}
\newtheorem{theorem/Def}[theorem]{Theorem/Definition}
\theoremstyle{definition}

\newtheorem{question}[theorem]{Question}
\newtheorem{Def}[theorem]{Definition}
\newtheorem{rem}[theorem]{Remark}
\newtheorem{exa}[theorem]{Example}

\def \Q {{\mathbb Q}}
\def \R {{\mathbb R}}
\def \N {{\mathbb N}}
\def \Z {{\mathbb Z}}

\def \D {{\pmb{D}}}

\def \O {{\mathcal O}}

\def \P {{\mathbb P}}

\def \T {{\mathbb T}}

\def \X {{ X_{\Sigma}}}

\def \p {{\tilde{\phi}}}

\def \f11 {{\frac{\log(u_v\bar{u}_v)\log(\nu_v\bar{\nu}_v)}{\log(u_v\bar{u}_v)+\log(\nu_v\bar{\nu}_v)}}}

\def \div {{\operatorname{div}}}

\def \vol {{ \operatorname{vol}}}

\def \Im {{ \operatorname{Im}}}

\def \relint {{ \operatorname{relint}}}

\def \max {{ \operatorname{max}}}

\def \conv {{ \operatorname{conv}}}
\def \convhull {{ \operatorname{convhull}}}
\def \MV {{ \operatorname{MV}}}

\def \Hom {{ \operatorname{Hom}}}

\def \We {{ \operatorname{We}}}

\def \Cl {{ \operatorname{Cl}}}
\def \Pic {{ \operatorname{Pic}}}
\def \Cox {{ \operatorname{Cox}}}

\def \prim {{ \operatorname{prim}}}

\def \timess {{ \operatorname{-times}}}

\def \deg {{ \operatorname{deg}}}

\def \BBig {{ \operatorname{Big}}}
\def \Nef {{ \operatorname{Nef}}}

\def \pr {{ \operatorname{pr}}}

\def \gcd {{ \operatorname{gcd}}}

\def \cone {{ \operatorname{cone}}}

\renewbibmacro{in:}{}

\usepackage{tgpagella}
\usepackage{tgheros}
\usepackage{eulervm}
\usepackage{tikz}
\usepackage{etoolbox}
\usetikzlibrary{arrows,matrix,positioning}
\tikzstyle{dmatrix}=[matrix of math nodes,row sep=2.5em, column sep=2.5em,
text height=1.5ex, text depth=0.25ex] 

\numberwithin{equation}{section}

\setparsizes{1em}{0.15\baselineskip plus .25\baselineskip}{1em plus 1fil}
\bibliography{bibliography_update}

\title{Intersection theory of toric $b$-divisors in toric varieties}
\author{Ana Mar\'ia Botero}
\date{}
\begin{document}

\maketitle

{\small{\begin{abstract}
We introduce toric $b$-divisors on complete smooth toric varieties
and a notion of integrability of such divisors. We show that under some
positivity assumptions toric $b$-divisors are integrable and that their
degree is given as the volume of a convex set. Moreover, we show that the
dimension of the space of global sections of a nef toric $b$-divisor
is equal to the number of lattice points in this convex set and we give
a Hilbert--Samuel type formula for its asymptotic growth. This generalizes
classical results for classical toric divisors on toric varieties. Finally, we relate convex bodies associated to $b$-divisors with Newton--Okounkov bodies. The main motivation for studying toric $b$-divisors is that they locally encode the singularities of the invariant metric on an automorphic line bundle over a toroidal compactification of a mixed Shimura variety of non-compact type.  
\end{abstract}}}

\section{Introduction}
One of the main goals of arithmetic geometry is to be able to \enquote{measure} the arithmetic complexity of an arithmetic object. This has led to the development of the theory of heights, numerical invariants defined via arithmetic intersection theory. Arithmetic intersection theory, which is also known as Arakelov theory, is a way to study varieties over rings of integers of number fields by putting smooth hermitian metrics on holomorphic vector bundles over the complex points of the variety. Its foundations go back to S.~J.~Arakelov in \cite{arakelov1,arakelov2}, where he developed the theory for arithmetic surfaces. This theory was generalized to higher dimensional arithmetic varieties by H. Gillet and C. Soul\'e in \cite{gillet-soule}. Later on, J.~I.~Burgos, J.~Kramer and U.~K\"uhn managed to weaken the smoothness condition of the metrics so as to allow so called log-singular metrics (\cite{BKK2,BKK1}). The key result used in order to develop the theory with this kind of singular metrics is the following theorem by D. Mumford \cite{MUM}:
\begin{theorem}
Every automorphic
line bundle on a pure open Shimura variety, equipped with
an invariant smooth metric, can be uniquely extended as a line
bundle on a toroidal compactification of the variety, in such a way
that the metric acquires only logarithmic singularities.
\end{theorem}
Applications of this theory have been found in the study of arithmetic cycles in pure Shimura varieties of non compact type such as in Kudla's programme that seeks to relate generating series of intersection numbers of arithmetic cycles to the Fourier coefficients of modular forms.  

A natural question is whether Mumford's result is valid also in the mixed Shimura setting. As was noticed by the authors in \cite{BKK}, Mumford's theorem is no longer valid in this situation: the invariant metric on an automorphic line bundle over a mixed Shimura variety of non-compact type acquires worse than logarithmic singularities over a toroidal compactification. As a consequence of the appearance of this new type of singularities, a \enquote{naive} extension of the metric along the boundary is not good as it depends on the choice of a compactification and does not satisfy Chern--Weil theory.

We address the geometric side of the problem in this and a number of subsequent articles. We propose a method for dealing with these new singularity types via convex geometric methods in the case of toroidal compactifications of universal abelian schemes and fibre products of those. 
The idea is to take \emph{all} possible toroidal compactifications into account and encode the singularity type in a \emph{toroidal $b$-divisor} (to be defined in a subsequent article). The correction terms needed to define a meaningful intersection theory should correspond to mixed degrees of such $b$-divisors.

In this article we develop an intersection theory of toric $b$-divisors on toric varieties. This sets the groundwork for the more general intersection theory of toroidal $b$-divisors on toroidal varieties. Moreover, we make a connection between convex bodies associated to toric $b$-divisors and Newton--Okounkov bodies induced from so called algebras of almost integral type and we give some applications of this identification. This gives an alternative insight into the birational geometry of toric varieties.

The outline of the paper is as follows. We start by setting some standard toric notation in Section \ref{notations}. Let $N \simeq \Z^n$ be a lattice and let $M = \Hom(N, \Z)$ be its dual lattice. Let $\X$ be the smooth and complete toric variety of dimension $n$ associated to a smooth and complete rational polyhedral fan $\Sigma \subseteq N_{\R} \, (\coloneqq N \otimes \R)$. Here, $N = M^{\vee}$ is the dual of the lattice $M$ of characters of the torus $\T \subseteq \X$. In Section \ref{section2} we define the central objects of study which are toric $b$-divisors on $\X$. We can think of toric $b$-divisors as tuples of toric $\Q$-Weil divisors, indexed by smooth subdivisions of $\Sigma$, compatible under push-forward.  We denote $b$-divisors in boldface notation $\D$ to distinguish them from classical divisors $D$. A toric $b$-divisor is said to be nef if for a cofinal subset of subdivisions, the corresponding constituents are nef. We characterize nef toric $b$-divisors by conical, $\Q$-concave, rational functions on the $\Q$-vector space $N_{\Q}$. In Section \ref{section3} we define an integrability notion of a collection of toric $b$-divisors. Our first result concerns integrability of a collection of nef toric $b$-divisors:
\begin{theorem}
Let $\D_1, \ldots,\D_n$ be a collection of nef toric $b$-divisors on $X_{\Sigma}$ and let $\phi_i \colon N_{\mathbb{Q}} \to \mathbb{Q}$ be the corresponding $\mathbb{Q}$-concave functions for $i = 1, \dotsc, n$.
     
      Then the functions $\phi_i$ extend to concave functions $\overline{\phi}_i \colon N_{\mathbb{R}} \to \mathbb{R}$.
      Moreover, the \emph{mixed degree} $\D_1\dotsm \D_n$ which is defined as the limit (in the sense of nets)
      \[
      \D_1\dotsm \D_n \coloneqq \lim_{\Sigma'} D_{1_{\Sigma'}} \dotsm D_{n_{\Sigma'}}
      \]
      exists, is finite, and is given by the mixed volume of the stability sets $\Delta_{\phi_i}$ of the concave functions $\overline{\phi}_i$, i.e.~we have 
      \[
      \D_1\dotsm \D_n = \MV\left(\Delta_{\overline{\phi}_1}, \dotsc, \Delta_{\overline{\phi}_n}\right).
      \]
     In particular, the \emph{degree} $\D^n$ of a nef toric $b$-divisor $\D$  with corresponding concave function $\overline{\phi}$ exists, is finite, and is given by 
      \[
      \D^n = n!\,\vol\left(\Delta_{\overline{\phi}}\right).
      \]

    \end{theorem}
    In Section \ref{section4} we consider the case of smooth and complete toric surfaces. Using basic number theory we show that integrability of a toric $b$-divisor on such a surface is equivalent to the convergence of a certain series ranging over pairs of relatively prime integers. 
    In Section \ref{section5} we define the space of global sections $H^0(X_{\Sigma}, \D)$ of a toric $b$-divisor $\D$. The following is the main result of this section.
    \begin{theorem}
      Let $\D$ be a toric $b$-divisor. There is a canonically defined convex set $\Delta_{\D}$ inducing an isomorphism of finite dimensional vector spaces
       \[
H^0(X_{\Sigma}, \D) \simeq \bigoplus_{m \in M \cap \Delta_{\D}} \mathbb{C} \cdot\chi^m.
\]
Moreover, if $\D$ is nef and $\phi_{\D}$ is its associated concave function on $N_{\R}$ then $\Delta_{\D} = \Delta_{\phi_{\D}}$ and the following Hilbert--Samuel type formula holds true:
\[
\D^n = \lim_{\ell \to \infty}\frac{h^0(X_{\Sigma}, \ell\D)}{\ell^n/n!}.
\]  
\end{theorem}
Finally, in Section \ref{section6} we identify the convex bodies arising from toric $b$-divisors $\Delta_{\D}$ with convex Okounkov bodies $\Delta_A \subseteq \R^n$ attached to graded algebras of almost integral type $A$ defined by K. Kaveh and A.~G. Khovanskii in \cite{KK}.
\begin{theorem}
Let $\D$ be a nef toric $b$-divisor on $\X$. Then the graded algebra 
\[
A_{\D} = \bigoplus_{\ell \geq 0}H^0(\X, \ell \D)t^{\ell}
\]
 is an algebra of almost integral type. Moreover, let $\Delta_{A_{\D}} \subseteq \R^n$ be the Newton--Okounkov body attached to $A_{\D}$. Then there is an identification $M_{\R} \simeq \R^n$ such that 
\[
\Delta_{A_{\D}} = \Delta_{\D}.
\]
\end{theorem}
As an application, in the big and nef case, we also construct a global $b$-convex body generalizing the global Okounkov body associated to a big divisor constructed in  \cite{LM} to the b-setting.\\[5mm]
\textbf{Acknowledgements} This project was supported by the IRTG 1800 on Moduli and Automorphic Forms and by the Berlin Mathematical School. I would like to thank J.~I.~Burgos for many illuminating discussions and numerous corrections and comments on earlier drafts. I also thank my PhD adivisor J.~Kramer and my PhD co-adivisor R. de Jong for helpful discussions and for suggesting this topic to me.

\section{Notation and basic facts}\label{notations}
Let $k$ be an algebraically closed field of arbitrary characteristic
and let $\T \simeq \mathbb{G}_{\text{m}}^n$ be an algebraic torus over $k$. An $n$-dimensional \emph{toric variety} is a normal $k$-variety with an open subvariety isomorphic to the torus $\T$ and such that the action of the torus on itself can be extended to a regular action on $X$.

We start by recalling some standard notation and basic facts from toric geometry. Let $N \simeq \Z^n$ be a lattice and let $M = \Hom(N, \Z)$ be its dual lattice. For any ring $R$, we denote by $N_{R}$ or $M_{R}$ the tensor product $N\otimes_{\Z}R$ or $M\otimes_{\Z}R$, respectively. A \emph{strongly convex rational polyhedral cone} $\sigma$ in $N_{\R}$ is a strongly convex polyhedral cone generated by vectors in $N$. We will use $\tau \leq \sigma$ to denote that $\tau$ is a face of $\sigma$.

A \emph{rational polyhedral fan} $\Sigma \subseteq N_{\R}$ is a collection of strongly convex rational polyhedral cones in $N_{\R}$ closed under \enquote{$\leq$} and such that every two cones in $\Sigma$ intersect along a common face. We always assume a fan $\Sigma$ to be non-degenerate, i.e.~it is not contained in any proper subspace of $N_{\R}$. We denote by $\Sigma(d)$ the set of $d$-dimensional cones in $\Sigma$. From now on, we will just say \emph{fan} and \emph{cone} even though we mean \enquote{rational polyhedral fan} and \enquote{strongly convex rational polyhedral cone}, respectively.

A fan $\Sigma$ is said to be \emph{complete} if the set $|\Sigma| \coloneqq \bigcup_{\sigma \in \Sigma}\sigma$ is the whole of $N_{\R}$. The set $|\Sigma|$ is called the \emph{support} of the fan $\Sigma$. The fan $\Sigma$ is said to be \emph{smooth} if  each cone $\sigma \in \Sigma$ is \emph{smooth}, i.e.~if $\sigma$ is spanned by part of a $\Z$-basis for $N$. A \emph{subdivision} of a fan $\Sigma \subseteq N_{\R}$ is a fan $\Sigma' \subseteq N_{\R}$ satisfying $|\Sigma| = |\Sigma'|$ and such that every cone of $\Sigma'$ is contained in a cone of $\Sigma$. 

Given a lattice $N$ of rank $n$ and a fan $\Sigma \subseteq N_{\R}$, it is a known fact that $\Sigma$ specifies a toric variety of dimension $n$, which we will denote by $\X$, with dense torus orbit $\T_N \subseteq \X$. The toric variety $\X$ is complete or smooth if and only if the corresponding fan $\Sigma$ is complete or smooth, respectively. The lattices $M$ and $N$ have natural geometric interpretations, namely $M$ is the character lattice of the torus $\T_N$ and its dual lattice $N$ is the lattice of one-parameter subgroups of the torus $\T_N$. We will denote characters by $\chi^m$ for elements $m$ in $M$. We will usually just write $\T$ for $\T_{N}$. We will denote by $\langle  \cdot _, \cdot \rangle$ the pairing between the dual lattices $M$ and $N$.  

For a given fan $\Sigma$, the cones of $\Sigma$ induce a stratification of $\X$ by torus orbits. As a particular case of the above correspondence, we get a bijective correspondence between the set $\Sigma(1)$ of rays of the fan $\Sigma$ and the set of prime toric Weil divisors. We will denote by $D_{\tau}$ the prime toric divisor associated to a ray $\tau \in \Sigma(1)$. This bijection induces a correspondence between toric $\Q$-Cartier divisors and so called  \emph{virtual support functions} which are continuous functions $\psi \colon |\Sigma| \to \R$ such that, for every cone $\sigma \in \Sigma$, there exists an $m_{\sigma} \in M_{\Q}$ with $\psi(u) = \langle m_{\sigma}, u \rangle$ for all $u \in \sigma$. Moreover, convexity properties of these functions encode positivity properties of the corresponding toric $\Q$-Cartier divisors: a function $f \colon N_{\R} \to \underline{\R} \; \left(\coloneqq \R \cup \{-\infty\}\right)$ is \emph{concave} if for all $x,y \in N_{\R}$, the following inequality 
\[
f(tx + (1-t)y) \geq tf(x) + (1-t)f(y) 
\]
is satisfied for $0 \leq t \leq 1$ and for all $x,y \in N_{\R}$. The correspondence between toric $\Q$-Cartier divisors and virtual support functions induces a bijective correspondence between concave virtual support functions and nef toric divisors.

A virtual support function $\psi \colon |\Sigma| \to \R$ is said to be \emph{
strictly concave} if it is concave and if for every $n$-dimensional cone $\sigma \in \Sigma(n)$ satisfies 
\[
\psi(u) =\langle m_{\sigma} , u \rangle \quad \text{iff} \quad u \in \sigma.
\]

The correspondence between toric $\Q$-Cartier divisors and virtual support functions induces a bijective correspondence between strictly concave virtual support functions and ample toric divisors.   

Assume that the fan $\Sigma$ is complete. A rational polytope $P \subseteq M_{\R}$ is the convex hull of finitely many elements in $M_{\Q}$.
We have that to any toric $\Q$-Weil divisor $D = \sum_{\tau \in \Sigma(1)}a_{\tau}D_{\tau}$ on the complete toric variety $\X$, one can associate the rational polytope
\[
P_D \coloneqq \left\{m \in M_{\R} \, | \, \langle m, v_{\tau} \rangle \geq -a_{\tau} \right\},
\]
where $v_{\tau}$ is the primitive vector spanning the ray $\tau$. 
If $D$ is $\Q$-Cartier, then $P_D$ encodes the global sections of the associated line bundle $\O(D)$. Indeed, in this case, we have that
\[
H^0\left(\X, \O(D)\right) \simeq \bigoplus_{m \in P_D \cap M}k \cdot \chi^{-m}.
\]
If $D$ is moreover nef, then $P_D = \conv\left(\left\{m_{\sigma}\right\}_{\sigma \in \Sigma(n)}\right) \subseteq M_{\R}$, where the $m_\sigma$ are the vectors determining the Cartier data of $D$. In this case, the degree of $D$ is given by
\[
D^n = n!\,\vol\left(P_D\right),
\]
where the operator \enquote{$\vol$} stands for the volume computed with respect to the Haar measure on $M_{\R}$ normalized so that the lattice $M$ has covolume 1. 

In terms of virtual support functions we have the following: let $f \colon N_{\R} \to \underline{\R}$ be a concave function. The \emph{Legendre--Fenchel dual} of $f$ is the function $f^{\vee} \colon M_{\R} \to \underline{\R}$ defined by the assignment 
\[
m \mapsto \text{inf}_{v \in N_{\R}}(\langle m,v \rangle -f(v)).
\]
The \emph{stability set} of $f$ is defined to be the convex set
\[
\Delta_{f} \coloneqq \text{dom}(f^{\vee}) = \{m \in M_{\R}\, |\, \langle m,v \rangle -f(v) \text { is bounded below for all }v \in N_{\R} \}.
\]
In particular, if $f =\psi_{D}$ is the concave virtual support function associated to a nef toric $\Q$-Cartier divisor $D$, then we have that $\Delta_{\psi} = P_D$ and hence 
\[
D^n = n!\vol\left(P_D\right) = n!\vol\left(\Delta_{\psi}\right).
\]
One can generalize the above fact about the top self intersection number of a nef toric $\Q$-Cartier divisor to the mixed top intersection number of a collection of toric nef $\Q$-Cartier divisors $D_1, \dotsc, D_n$. Indeed, the \emph{mixed volume} of a collection of convex sets $K_1, \dotsc, K_n$ is defined by 
\[
\MV\left(K_1, \dotsc,K_n\right) \coloneqq \sum_{j=1}^n(-1)^{n-j}\sum_{1\leq i_1<\dotsb <i_j\leq n}\vol\left(K_{i_1}+ \dotsb + K_{i_j}\right), 
\]
where the \enquote{$+$} refers to \emph{Minkowski addition} of convex sets. Then, the mixed degree of $D_1, \dotsc, D_n$ is equal to the mixed volume of the corresponding polytopes, i.e.~we have
\[
D_1\dotsm D_n = \MV\left(P_{D_1}, \dotsc ,P_{D_n}\right).
\] 
We refer to \cite{CLS} and to \cite{ful} for a more detailed introduction to toric geometry.

\section{$b$-divisors on toric varieties}\label{section2}
Throughout this article we will fix an algebraically closed field $k$, a lattice $N$ of rank $n$ and a smooth, complete fan $\Sigma \subseteq N_{\R}$. We will denote by $\X$ the corresponding smooth and complete toric variety with dense open torus $\T = \T_{N} \simeq \left(k^*\right)^n$. 

In this section we define toric $b$-divisors on the toric variety $\X$ as a tower of toric Cartier $\Q$-divisors indexed over all toric proper birational morphisms from smooth and complete toric varieties to $\X$ and satisfying some compatibility condition. Moreover, as it is usual in toric geometry where geometric objects correspond to combinatorial ones, we will show that toric $b$-divisors correspond to conical, $\Q$-valued functions on the space $N_{\Q}$.

Let $\T\text{-Ca}\left(X_{\Sigma}\right)$ and $\T\text{-We}\left(X_{\Sigma}\right)$ be the space of toric Cartier and of toric Weil divisors, respectively. The word \enquote{toric} means that they are invariant under the action of the torus. We consider Cartier and Weil divisors with $\Q$-coefficients, i.e.~the spaces $\T\text{-Ca}\left(X_{\Sigma}\right)_{\Q}$ and $\T\text{-We}\left(X_{\Sigma}\right)_{\Q}$, respectively.
\begin{Def} The set $R\left(\Sigma\right)$ is defined to be the collection of smooth fans subdividing $\Sigma$. This forms a directed set under the following partial order 
\[
\Sigma'' \geq \Sigma' \quad \text{iff} \quad \text{$\Sigma''$ is a smooth subdivision of $\Sigma'$}.
\]
The \emph{toric Riemann--Zariski space} of $\X$ is defined as the inverse limit
\[
\mathfrak{X}_{\Sigma} \coloneqq \varprojlim_{\Sigma' \in R\left(\Sigma\right)}X_{\Sigma'}
\] with maps given by the toric proper birational morphisms $X_{\Sigma''} \to X_{\Sigma'}$ induced whenever $\Sigma'' \geq \Sigma'$. 
\end{Def}
\begin{rem}
It is a standard fact in toric geometry that we have a bijective correspondence between fans $\Sigma' \in R\left(\Sigma\right)$ and pairs $(X_{\Sigma'}, \pi_{\Sigma'})$ of smooth and complete toric varieties $X_{\Sigma'}$ together with toric proper birational morphisms $\pi_{\Sigma'} \colon X_{\Sigma'} \to \X$, up to isomorphism (see e.g. \cite[Theorem~3.3.4 and 3.4.11]{CLS}). 
\end{rem}  
\begin{Def}\label{def:toricbdivisor}
We define the group of \emph{toric Cartier $b$-divisors} on $\X$ as the direct limit
\[
\text{Ca}(\mathfrak{X}_{\Sigma})_{\Q} \coloneqq \varinjlim_{\Sigma' \in R\left(\Sigma\right)} \T\text{-Ca}(X_{\Sigma'})_{\Q} 
\]
with maps given by the pull-back map of toric $\Q$-Cartier divisors on $\X$. 
\noindent The group of \emph{toric Weil $b$-divisors} is defined as the inverse limit
\[
\text{We}(\mathfrak{X}_{\Sigma})_{\Q} \coloneqq \varprojlim_{\Sigma' \in R\left(\Sigma\right)} \T\text{-We}(X_{\Sigma'})_{\Q} 
\]
with maps given by the push-forward map of toric $\Q$-Weil divisors on $\X$. 
\end{Def}
\begin{rem} Our definition of toric Cartier and toric Weil $b$-divisors is inspired by Shokurov's $b$-divisors for which \cite{BFF} constitutes a thorough reference. The difference is that in this article we consider only \emph{toric} proper birational models which are somewhat governed by the combinatorics, whereas Shokurov's $b$-divisors are indexed over all proper birational models. Also, note that the toric variety $\X$ is smooth, hence toric Weil divisors are also Cartier.    
\end{rem}
\noindent
\textbf{Notation.} Since we will mostly deal with $\Q$-coefficients, we will omit the $\Q$ from our notations ($\Q$-coefficients always being implicit, unless stated otherwise). We will denote $b$-divisors in bold notation $\D$ to distinguish them from classical divisors $D$.  
\begin{rem}
We make the following remarks.
\begin{enumerate}
\item It follows from the definitions that a toric Weil $b$-divisor $\D$ on $\X$ consists of a tuple of toric Cartier $\Q$-divisors 
\begin{align}\label{tuple}
\D = \left(D_{\Sigma'}\right)_{\Sigma' \in R(\Sigma)},
\end{align}
where $D_{\Sigma'} \in \T\text{-Ca}(X_{\Sigma'})$, which are compatible under push-forward, i.e.~such that we have $\pi_*D_{\Sigma''}=D_{\Sigma'}$ whenever $\Sigma'' \geq \Sigma'$ and $\pi \colon X_{\Sigma''} \to X_{\Sigma'}$ is the corresponding toric morphism. For every $\Sigma' \in R(\Sigma)$, we say that $D_{\Sigma'}$ is the \emph{incarnation} of $\D$ on the model $X_{\Sigma'}$. On the other hand, a toric Cartier $b$-divisor $\pmb{E}$ on $\X$ is a toric Weil $b$-divisor 
\begin{align}\label{tuple1}
\pmb{E} = \left(E_{\Sigma'}\right)_{\Sigma' \in R(\Sigma)},
\end{align}
for which there is a model $X_{\Sigma'}$ for a $\Sigma' \in R\left(\Sigma\right)$ such that for every other higher model $X_{\Sigma''}$, where $\Sigma'' \geq \Sigma'$ in $R\left(\Sigma\right)$, the incarnation $D_{\Sigma''}$ is the pullback of $D_{\Sigma'}$ on $X_{\Sigma'}$. In this case, we say that $\pmb{E}$ is \emph{determined} by $\Sigma'$. A toric Weil or Cartier $b$-divisor $\D$ or $\pmb{E}$ is assumed to come always with a tuple as in (\ref{tuple}) or (\ref{tuple1}), respectively.   
\item Since our fans are assumed to be smooth, we may identify Cartier and Weil toric divisors. 
\item We clearly have the containment $\text{Ca}(\mathfrak{X}_{\Sigma})_{\Q} \subseteq \text{We}(\mathfrak{X}_{\Sigma})_{\Q}$. Henceforth, we will refer to a \emph{toric Weil $b$-divisor} simply as a \emph{toric $b$-divisor}.
\end{enumerate}
\end{rem}
The next lemma gives the combinatorial analogues of toric $b$-divisors. Before stating it, we give some definitions.
\begin{Def}\label{pushf}
Let $\Sigma' \in R(\Sigma)$ and let $\psi_{\Sigma'}$ be a virtual support function on the toric variety $X_{\Sigma'}$. Then for every fan $\tilde{\Sigma} \in R(\Sigma)$ such that $\tilde{\Sigma} \leq \Sigma'$ and inducing $\pi \colon X_{\Sigma'} \to X_{\tilde{\Sigma}}$, we will denote by $\pi_*(\psi_{\Sigma'})$ the piecewise linear function on $\tilde{\Sigma}$ defined by the values of $\psi_{\Sigma'}$ on the rays of $\tilde{\Sigma}$.   
\end{Def}
\begin{Def}
A function $f \colon N_{\R}\to \R$ is called \emph{conical} if $f(\lambda x) = \lambda f(x)$ for all $\lambda \in \R_{\geq 0}$. 
\end{Def}
\begin{Def} The set $N^{\prim}$ is defined to be the subset of $N$ consisting of primitive lattice elements. 
\end{Def}
\begin{lemma}\label{equi} There is a natural bijective correspondence between the following objects:
\begin{enumerate}
\item\label{cc} Toric $b$-divisors on $\X$.
\item\label{cccc} Collections $\left\{\psi_{\Sigma'}\right\}_{\Sigma' \in R\left(\Sigma\right)}$ of $\R$-valued functions on $N_{\R}$ satisfying the following two conditions:
\begin{enumerate}
\item For each $\Sigma' \in R\left(\Sigma\right)$, the function $\psi_{\Sigma'}$ is a virtual support function for $\Sigma'$.
\item\label{compa} For every toric birational morphism $\pi \colon X_{\Sigma''} \to X_{\Sigma'}$ induced by a smooth subdivision $\Sigma'' \geq \Sigma'$ we have $\pi_*(\psi_{\Sigma''}) = \psi_{\Sigma'}$. 
\end{enumerate}
\item\label{ccc} Conical functions $\tilde{\phi} \colon N_{\Q} \to \Q$.
\end{enumerate}
\end{lemma}
\begin{proof}
The bijective correspondence between (\ref{cc}) and (\ref{cccc}) comes from the general theory of toric divisors and the fact that $b$-divisors are defined as a projective limit, hence requiring the compatibility condition in (\ref{compa}). 

To see that (\ref{cccc}) implies (\ref{ccc}) let $\left\{\psi_{\Sigma'}\right\}_{\Sigma'\in R\left(\Sigma\right)}$ be a collection as in (\ref{cccc}). Note that a conical function on $N_{\Q}$ is determined by it values on $N^{\prim}$. Hence, let $v$ be in $N^{\text{prim}}$ and consider $\tau_v$ the ray spanned by $v$.  Let $\Sigma' \in R\left(\Sigma\right)$ be any fan containing $\tau_v$. We define $\tilde{\phi}(v) \coloneqq \psi_{\Sigma'}(v)$. To see that this is well defined, let $\Sigma'' \in R\left(\Sigma\right)$ be any other fan containing $\tau_v$. Let $\Sigma'''\in R\left(\Sigma\right)$ be a common refinement of $\Sigma'$ and $\Sigma''$. Then, clearly, $\tau_v \in \Sigma'''$ and, denoting by $\pi'\colon X_{\Sigma'''} \to X_{\Sigma'}$ and by $\pi''\colon X_{\Sigma'''} \to X_{\Sigma''}$ the corresponding morphisms, we have \[\psi_{\Sigma'}(v) = \pi'_*\psi_{\Sigma'''}(v) = \pi''_*\psi_{\Sigma'''}(v) = \psi_{\Sigma''}(v).\]
Now, to see that (\ref{ccc}) implies (\ref{cc}) suppose that we are given a function $\tilde{\phi} \colon N^{\text{prim}} \to \Q$. This induces for every $\Sigma' \in R\left(\Sigma\right)$ a toric Weil divisor 
\[
D_{\tilde{\phi},\Sigma'}\coloneqq \sum_{\tau \in \Sigma'(1)}a_{\tau}D_{\tau}
\]
on $X_{\Sigma'}$  by setting $a_{\tau} = -\tilde{\phi}\left(v_{\tau}\right)$ and hence, by smoothness, a toric Cartier divisor on each $X_{\Sigma'}$. The fact that this collection forms a toric $b$-divisor follows from the fact that for every toric morphsim $\pi \colon X_{\Sigma''} \to X_{\Sigma'}$ coming from a regular subdivision $\Sigma'' \geq \Sigma'$, we have \[\pi_*\left(D_{\tilde{\phi},\Sigma''}\right) = \pi_* \left(\sum_{\tau \in \Sigma''(1)}a_{\tau}D_{\tau}\right) = \sum_{\tau \in \Sigma'(1)}a_{\tau}D_{\tau} = D_{\tilde{\phi},\Sigma'}.\qedhere
\]
\end{proof} 
Before we end this section, we define a special class of smooth subdivisions of the fan~$\Sigma$. The following definitions are taken taken from \cite[Section~3.1]{CLS}. 
\begin{Def}\label{barycentric1} Let $\sigma$ be a $d$-dimensional smooth cone in $\Sigma$. Let $\{v_1, \dotsc, v_d\}$ be generators of $\sigma$ forming part of a $\Z$-basis of $N$. The \emph{barycenter} $v_{\sigma}$ of $\sigma$ is defined to be the element
\[
v_{\sigma} \coloneqq \sum_{i=1}^{d}v_i.
\]
\end{Def}
\begin{Def}\label{barycentric2}
Given a fan $\Sigma$ and a smooth cone $\sigma \in \Sigma$, the \emph{barycentric subdivision of $\Sigma$ corresponding to the primitive element $v_{\sigma}$}, denoted by $\Sigma^*(v_{\sigma})$, is the fan consisting of the following cones:
\begin{itemize}
\item $\tau \in \Sigma$, where $v_{\sigma} \notin \tau$.
\item $\text{cone}(\tau, v_{\sigma})$ where $v_{\sigma} \notin \tau \in \Sigma$ and $\{v_{\sigma} \cup \tau\}$ is contained in a cone $\sigma' \in \Sigma$. 
\end{itemize} 
If $\Sigma$ is smooth, then $\Sigma^*(v_{\sigma})$ is a smooth refinement of $\Sigma$. We also refer to $\Sigma^*(v_{\sigma})$ as the \emph{barycentric subdivision of $\Sigma$ at the cone $\sigma$}. 
\end{Def}
\begin{exa}
Consider the (non-complete) fan $\Sigma \subseteq \R^3$ consisting of the cone $\sigma = \R^3_{\geq 0}$ together with all of its faces. Figure \ref{flor} shows a $2$-dimensional picture of two star subdivisions: one at the barycenter of $\sigma$, and one at the barycenter of the two dimensional cone $\tau$.   
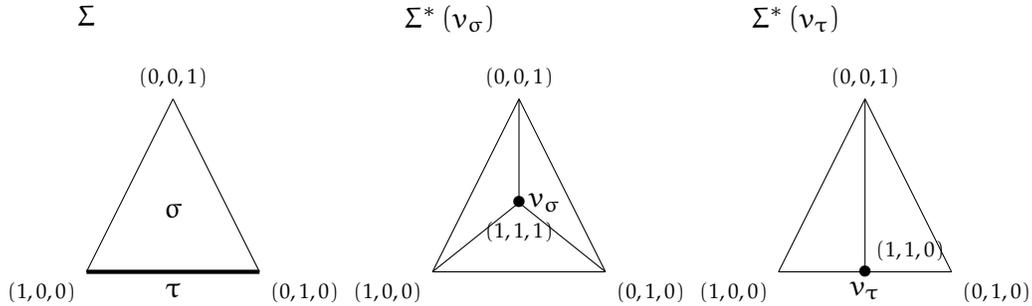
\begin{figure}[H]
\begin{center}
\begin{tikzpicture}[scale=1.15]
    \draw[ultra thick]  (0,0) -- (2,0) node [midway,below]{$\tau$};
    \draw (0,0) -- (1,2);
    \draw (1,2) -- (2,0);
    \draw (1,0.7) node{$\sigma$};
    \draw (0,0) node[below left]{\scriptsize{$(1,0,0)$}};  
    \draw (2,0) node[below right]{\scriptsize{$(0,1,0)$}}; 
    \draw (1,2) node[above]{\scriptsize{$(0,0,1)$}}; 
    \draw (0,3.2) node[below]{$\Sigma$};
     \draw  (4,0) -- (6,0);
     \draw (4.2,3.2) node[below]{$\Sigma^*\left(v_{\sigma}\right)$};
    \draw (4,0) -- (5,2);
    \draw (5,2) -- (6,0);
    \draw (5,0.8) node[right]{$v_{\sigma}$};
    \draw (4,0)--(5,0.8);
    \draw (6,0)--(5,0.8);
    \draw (5,2)--(5,0.8);
    \draw (5,0.8) node{$\bullet$};
    \draw (4,0) node[below left]{\scriptsize{$(1,0,0)$}};  
    \draw (6,0) node[below right]{\scriptsize{$(0,1,0)$}}; 
    \draw (5,2) node[above]{\scriptsize{$(0,0,1)$}}; 
    \draw (5,0.7) node[below]{\scriptsize{$(1,1,1)$}}; 
    \draw  (8,0) -- (10,0) node[midway,below]{$v_{\tau}$};
    \draw (8.2,3.2) node[below]{$\Sigma^*\left(v_{\tau}\right)$};
    \draw (8,0) -- (9,2);
    \draw (9,2) -- (10,0);
    \draw (9,0)--(9,2);
    \draw (9,0) node{$\bullet$};
      \draw (8,0) node[below left]{\scriptsize{$(1,0,0)$}};  
    \draw (10,0) node[below right]{\scriptsize{$(0,1,0)$}}; 
    \draw (9,2) node[above]{\scriptsize{$(0,0,1)$}}; 
    \draw (9,0) node[above right]{\scriptsize{$(1,1,0)$}}; 
    
\end{tikzpicture}
\end{center}
\caption{Star subdivisions of $\R^3_{\geq 0}$}\label{flor}
\end{figure}
\end{exa}
\begin{rem}\label{starcofinite}
We have the following two remarks.
\begin{enumerate}
\item Given any smooth subdivision $\Sigma'$ of $\Sigma$ we can always find a smooth fan  $\tilde{\Sigma}$ obtained by a sequence of barycentric subdivisions of $\Sigma$ which  dominates $\Sigma'$. This follows from the fact that barycentric subdivisions correspond to blow ups of toric varieties along smooth torus invariant centers (\cite[Proposition~3.3.15]{CLS}) and a particular case of the following De Concini--Procesi Theorem (\cite[Theorem~4.1]{concini}).
\begin{theorem} Let $X$ and $X'$ be smooth toric varieties with a toric birational map (not necessarily everywhere defined) $f\colon X \dashrightarrow X'$. Then, there exists a smooth toric variety $\tilde{X}$ obtained from $X$ by a finite sequence $\pi \colon \tilde{X} \to X$ of blow-ups along smooth $2$-codimensional invariant centers, such that $f \circ \pi$ extends to a regular map.
\end{theorem}
Hence, the set of barycentric subdivisions of $\Sigma$ along smooth cones is cofinal in $R(\Sigma)$.
\item\label{aback}
We can give $R\left(\Sigma\right)$ the structure of a connected, oriented graph in the following way: let $\Sigma' \in R\left(\Sigma\right) $. If $\Sigma'' \in R(\Sigma)$ is such that $\Sigma'' = \Sigma'^*(v_{\sigma})$ for a cone $\sigma \in \Sigma'$, then we put an arrow $\Sigma'' \rightarrow \Sigma'$. 
\end{enumerate}
\end{rem}
To finalize this section, note that Lemma \ref{equi} implies that a priori we don't have any control over the coefficients of the prime toric divisors appearing in the different incarnations of a toric $b$-divisor as we move along the toric birational models. This is why, in general, in order to prove integrability results about toric $b$-divisors, we need to impose some positivity condition. It turns out, as it will become clear in the following section, that the right positivity condition is \emph{nefness}.

\section{Integrability of toric $b$-divisors}\label{section3}
In this section we define the mixed degree of a toric $b$-divisor. We further investigate the question regarding both necessary and sufficient conditions for integrability (i.e.~existence and finiteness of the mixed degree). We give a partial answer to this question in the general case of dimension $n$ and, in the follo\-wing section, a complete answer in the case of toric surfaces. Specifically, in the $n$-dimensional case, we show that any collection of $n$ nef toric $b$-divisors (and any difference of such) is integrable and give a formula for its mixed degree as the mixed volume of bounded convex sets (and differences of such). As a corollary we get a Brunn--Minkowski type inequality for the degree of nef toric $b$-divisors. 

Before introducing the definition of the \emph{mixed degree} of a collection of toric $b$-divisors, note that there is an obvious intersection pairing 
\[
\underbrace{\text{Ca}(\mathfrak{X}_{\Sigma}) \times \dotsm \times \text{Ca}(\mathfrak{X}_{\Sigma})}_{n\timess} \longrightarrow \Q
\]
of toric Cartier $b$-divisors defined as follows: if $(\pmb{E}_1, \dotsc ,\pmb{E}_2) \in \text{Ca}(\mathfrak{X}_{\Sigma}) \times \dotsm \times \text{Ca}(\mathfrak{X}_{\Sigma})$, then let $\Sigma' \in R\left(\Sigma\right)$ be sufficiently large such that $\pmb{E}_i$ is determined on $\Sigma'$ for all $i=1, \dotsc, n$. We then set  
\[
\pmb{E}_1\dotsm \pmb{E}_n \coloneqq E_{1, \Sigma'} \dotsm E_{n, \Sigma'}.
\]
This is independent of the choice of a common determination $\Sigma'$ and hence the pairing given by 
\[
\left(\pmb{E}_1,\dotsc, \pmb{E}_n\right) \longmapsto \pmb{E}_1\dotsm \pmb{E}_n
\] is well defined. Moreover, we can extend the above intersection pairing to a pairing
\[
\underbrace{\text{Ca}(\mathfrak{X}_{\Sigma}) \times \dotsm \times \text{Ca}(\mathfrak{X}_{\Sigma}) \times \text{We}(\mathfrak{X}_{\Sigma})}_{n \timess} \longrightarrow \Q
\]
in the following way: if $(\pmb{E}_1, \dotsc ,\pmb{E}_{n-1}, \D) \in \text{Ca}(\mathfrak{X}_{\Sigma}) \times \dotsm \times\text{Ca}(\mathfrak{X}_{\Sigma}) \times \text{We}(\mathfrak{X}_{\Sigma})$, we let $\Sigma'$ be a common determination of $\pmb{E}_i$ for $i =1, \dotsc, n-1$. We then set 
\[
\pmb{E}_1 \dotsm \pmb{E}_{n-1} \cdot \D \coloneqq E_{1, \Sigma'} \dotsm E_{n-1, \Sigma'}\cdot D_{\Sigma'}.
\]
Using the projection formula, one can see that the above pairing is independent of the choice of a common determination $\Sigma'$ and hence is well defined. Indeed, let $\Sigma'' \in R(\Sigma)$ be another common determination of the $\pmb{E}_i$'s. W.l.o.g. assume that $\Sigma'' \geq \Sigma'$ and let $\pi_{\Sigma''} \colon X_{\Sigma''} \to X_{\Sigma'}$ the induced toric, proper, birational morphism. We get 
\begin{align*}
E_{1, \Sigma''} \dotsm E_{n-1, \Sigma''}\cdot D_{\Sigma''} &= \pi_{\Sigma''}^*\left(E_{1, \Sigma'}\right) \dotsm \pi_{\Sigma''}^*\left(E_{n-1, \Sigma'}\right)\cdot D_{\Sigma''} \\ &= \pi_{\Sigma''}^*\left(E_{1, \Sigma'} \dotsm E_{n-1, \Sigma'}\right)\cdot D_{\Sigma''} \\ &= E_{1, \Sigma'} \dotsm E_{n-1, \Sigma'}\cdot \pi_{\Sigma''*}\left(D_{\Sigma''}\right) \\ &= E_{1, \Sigma'} \dotsm E_{n-1, \Sigma'}\cdot D_{\Sigma'},
\end{align*} 
  where we have used the projection formula to pass from the second to the third line. However, in general, there is no obvious way to define an intersection pairing for more than one element in $\We(\mathfrak{X}_{\Sigma})$. The following definition is inspired by the idea of \enquote{integrability} of a function (see \cite[Definition~3.17]{BKK}). Recall that $R\left(\Sigma\right)$ is a directed set. In particular, it is a \emph{net} and we have a well defined notion of limit.     
\begin{Def} Let $\D_1, \dotsc, \D_n$ be a collection of toric $b$-divisors. We define its \emph{mixed degree} to be the limit 
\[
\D_1 \dotsm \D_n \coloneqq \lim_{\Sigma' \in R\left(\Sigma\right)} D_{1,\Sigma'} \dotsm D_{n,\Sigma'},
\] 
if it exists and is finite. If $\D= \D_1 = \dotsb =\D_n$, then we call 
\[
\D^n = \lim_{\Sigma' \in R\left(\Sigma\right)} D_{\Sigma'}^n,
\]
if the limit exists and is finite, the \emph{degree} of the $b$-divisor $\D$. A toric $b$-divisor whose degree exists is called \emph{integrable}.
\end{Def}
One of the main results of this section is that the mixed degree of a collection of \emph{nef} toric $b$-divisors (see Definition \ref{nefdef}) exists. In particular, a nef toric $b$-divisor is automatically integrable.

\subsection*{Integrability of nef toric $b$-divisors}
 \begin{Def}\label{nefdef} A toric $b$-divisor $\D =  (D_{\Sigma'})_{\Sigma' \in R\left(\Sigma\right)}$ is called \emph{nef} if there exists a cofinal subset $S \subseteq R(\Sigma)$ such that $D_{\Sigma'}$ is nef in $X_{\Sigma'}$ for all $\Sigma' \in S$.
\end{Def}

\begin{rem}\label{abab}
With notations as above, let $\psi_{\Sigma'} = \psi_{D_{\Sigma'}} \colon N_{\R} \to \R$ be the virtual support function corresponding to $D_{\Sigma'}$. Recall that, if $D_{\Sigma'} = \sum_{\tau \in \Sigma'(1)}a_{\tau}D_{\tau}$, then $\psi_{\Sigma'}$ is characterized by satisfying $\psi_{\Sigma'}(v_{\tau}) = -a_{\tau}$, where $v_{\tau}$ denotes the primitive vector spanning the ray $\tau$. Hence, if $\Sigma'' \geq \Sigma'$ and the map $\pi \colon X_{\Sigma''} \to X_{\Sigma'}$ is the associated toric morphism, then 
\[
D_{\Sigma''} \leq \pi^*D_{\Sigma'} \quad \text{ iff }\quad \psi_{\Sigma''}(v_{\tau}) \geq \psi_{\Sigma'}(v_{\tau}) 
\]
for every $\tau \in \Sigma''(1)$.    
\end{rem}
Now, recall that classical nef toric Cartier divisors are in bijective correspondence with virtual support functions which have the additional property of being \emph{concave}. Hence, using Remark \ref{abab}, one expects that monotonicity properties of nef toric divisors follow directly from such analogous properties of concave functions. Indeed, we have the following lemma.  
\begin{lemma}\label{con1} Let $\D$ be a nef toric $b$-divisor on $\X$ and let $\Sigma'$ be a fan in $R\left(\Sigma\right)$. Then, for a sufficiently large $\Sigma'' \geq \Sigma'$ in $R\left(\Sigma\right)$, we have that $D_{\Sigma''} \leq \pi_{\Sigma''}^*D_{\Sigma'}$. Here, the morphism $\pi_{\Sigma''}\colon X_{\Sigma''} \to X_{\Sigma'}$ denotes the induced toric, proper, birational morphism.  
\end{lemma} 
\begin{proof} 
We denote by $\relint(\sigma)$ the relative interior of a convex set $\sigma$. Let $\Sigma'' \geq \Sigma'$ be in $R\left(\Sigma\right)$ such that $D_{\Sigma''}$ is nef and hence $\psi_{\Sigma''}$ is concave (note that this is not necessarily the case for all $\Sigma'' \geq \Sigma$). By Remark \ref{abab}, and using the same notation therein, it suffices to show that for all $\tau \in \Sigma''(1)$ the inequality
\begin{align}\label{con}
\psi_{\Sigma''}(v_{\tau}) \geq \psi_{\Sigma'}(v_{\tau}) 
\end{align}
is satisfied. Now, if $\tau \in \Sigma'(1)$ then $\psi_{\Sigma''}(v_{\tau}) = \psi_{\Sigma'}(v_{\tau})$. Otherwise, assume that $\tau \in \Sigma''(1) \setminus \Sigma'(1)$. Let $\tau_1, \dotsc , \tau_r \in \Sigma'(1)$ such that $\sigma \coloneqq \text{cone}\langle \tau_1, \dotsc , \tau_r \rangle \in \Sigma'$ and $\tau \in \relint(\sigma)$. Then we can write 
\[
v_{\tau} = \sum_{i=1}^ra_i v_{\tau_i} \quad \text{with}\quad a_i > 0.
\]
Hence, we have 
\[
\psi_{\Sigma''}(v_{\tau}) = \psi_{\Sigma''}\left(\sum_{i=1}^ra_i v_{\tau_i}\right) \geq  \sum_{i=1}^ra_i\psi_{\Sigma''}(v_{\tau_i}) = \sum_{i=1}^ra_i\psi_{\Sigma'}(v_{\tau_i})= \psi_{\Sigma'}(v_{\tau}).
\]
Note that the second inequality follows by the concavity of $\psi_{\Sigma''}$. Thus, the inequality~(\ref{con}) is satisfied and thus the statement of the lemma follows.
\end{proof}
Next, we will show that this monotonicity property extends to intersection products. Recall the definition of the rational polytope $P_D$ associated to a toric divisor $D$ from Section~\ref{notations}. 
\begin{rem}\label{aaa}
Lemma \ref{con1} and its proof imply that if $\D = (D_{\Sigma'})_{\Sigma' \in R\left(\Sigma\right)}$ is a nef toric $b$-divisor, then for every fan $\Sigma'$ in $R\left(\Sigma\right)$, there exists a sufficiently large refinement $\Sigma'' \geq \Sigma'$ in $R\left(\Sigma\right)$ such that
\[
\Sigma'' \geq \Sigma' \quad \text{ iff }\quad P_{D''} \subseteq P_{D'}.
\]
\end{rem}
We can now state the aimed monotonicity property for intersection products of nef toric $b$-divisors. 
\begin{lemma}\label{mul}
Let $\D_1, \dotsc, \D_n$ be nef toric $b$-divisors. Then for every fan $\Sigma'$ in $R\left(\Sigma\right)$, there exists a sufficiently large refinement $\Sigma'' \geq \Sigma'$ in $R\left(\Sigma\right)$ such that the inequality 
\[
D_{1, \Sigma''} \dotsm D_{n, \Sigma''} \leq D_{1, \Sigma'} \dotsm D_{n, \Sigma'}
\]
holds true.
\end{lemma}
\begin{proof}
Recall from Section \ref{notations}, that the mixed degree of a collection of toric nef divisors $D_1, \dotsc, D_n$ is given by the mixed volume of their corresponding polytopes, i.e.~we have
\[
D_1\dotsm D_n  = \MV\left(P_{D_1}, \dotsc, P_{D_n}\right).  
\]
Hence, by Remark \ref{aaa}, for a large enough refinement $\Sigma'' \geq \Sigma'$ in $R(\Sigma)$, we have that
\[
D_{1, \Sigma''} \dotsm D_{n, \Sigma''} = \text{MV}\left(P_{D_{1,\Sigma''}}, \dotsc, P_{D_{n,\Sigma''}}\right) \leq\text{MV}\left(P_{D_{1,\Sigma'}}, \dotsc ,P_{D_{n,\Sigma'}}\right) = D_{1, \Sigma'} \dotsm D_{n, \Sigma'}.\qedhere
\]
\end{proof}
\begin{rem}\label{remark}
It follows from Lemma \ref{equi} that a nef toric $b$-divisor corresponds to a conical $\Q$-concave function $\tilde{\phi} \colon N_{\Q} \to \Q$. Here, by $\Q$-concavity we mean that for all convex combinations $\sum_{i=1}^r~a_i v_i \in N_{\Q}$, i.e. the $a_i$'s are in $\Q_{\geq 0}$ and satisfy $\sum_{i=1}^ra_i = 1$, the inequality
\[
\p\left(\sum_{i=1}^ra_i v_i\right) \geq \sum_{i=1}^ra_i\p\left(v_i\right) 
\]
holds.
\end{rem}
Before giving the main result of this section, we have a technical lemma. 
\begin{lemma}\label{lemma1}
Let $\p \colon N_{\Q} \to \Q$ be a $\Q$-concave function and let $\bar{x} \in N_{\R}$. Then we have 
\begin{enumerate}
\item\label{parteuno} There exist two positive real constants $C$ and $r $ such that $|\p(x)| \leq C$ for all $x \in B(\bar{x},r)$, where $B(\bar{x},r)~\coloneqq~\{x \in N_{\Q} \, \big{|} \, |x-\bar{x}| \leq r\}$.
\item\label{partedos} The function $\p$ is Lipschitz continuous at $\bar{x}$, i.e.~there exist two positive real constants $L$ and $r$ such that 
\[
|\p(x) - \p(x')| \leq L|x -x'|
\]
for all $x, x' \in B(\bar{x},r)$. 
\end{enumerate}
\end{lemma}
\begin{proof} In order to prove part \ref{parteuno} of the lemma, we consider an $n$-dimensional simplex $\Delta$ with vertices $x_0, \dotsc , x_n \in N_{\Q}$ containing $\bar{x}$ in its interior, i.e. 
\[
\bar{x} = \sum_{i=0}^n a_i x_i
\]
with $a_i \in \R_{>0}$ such that $\sum_i a_i = 1$. Since $\Delta$ is top-dimensional, it contains a ball $B(\bar{x},r)$ for some suitable small radius $r >0$. We show that the function $\p$ is bounded from below in $B(\bar{x},r)$ by the quantity $\min_{0 \leq i \leq n}\p(x_i)$. Indeed, let $x \in B(\bar{x},r)\subseteq \Delta \cap N_{\Q}$. Then $x$ is of the form $x = \sum_{i=0}^n b_ix_i$ with $b_i \in \Q_{\geq 0}$ satisfying $\sum_{i=0}^nb_i =1$. By the $\Q$-concavity, we get
\[
\p(x) = \p\left(\sum_{i=0}^n b_ix_i\right) \geq \sum_{i=0}^n b_i\p(x_i) \geq \min_{0 \leq i \leq n}\p(x_i).
\]
Letting $C' \coloneqq \min_{0 \leq i \leq n}\p(x_i)$, we have shown that $\p$ is bounded from below by $C'$ in $B_{\bar{x},r}$ as claimed. \\
Now, we prove that $\p$ is bounded from above in $B(\bar{x},r)$ for the same radius $r$. For this, let $x\in B(\bar{x},r)$ and set $x' = \bar{x} - (x-\bar{x}) = 2\bar{x}-x$. We get $|x'-\bar{x}|=|x-\bar{x}| \leq r$ and thus $x'\in B(\bar{x},r)$. Since $\bar{x}=\frac{1}{2}(x+x')$, the $\Q$-concavity of $\p$ shows that
\[
2\p(\bar{x}) \geq \p(x) + \p(x'). 
\]
Therefore, since $\p(x') \geq C'$, we get 
\[
\p(x) \leq 2\p(\bar{x}) - \p(x') \leq 2\p(\bar{x}) - C'.
\]
By taking $C \coloneqq \max\left\{C',2\p(\bar{x}) - C'\right\}$, the proof of part \ref{parteuno} is complete. 
Now, let us see that part \ref{partedos} of the lemma is a direct consequence of part \ref{parteuno}. Indeed, we show that $\p$ is Lipschitz continuous in the neighborhood $B(\bar{x},r/2)$. Let $x, x' \in B(\bar{x},r/2)$ with $x \neq x'$. There exists an $x'' \in N_{\Q}$ such that the convex linear combination 
\[
x' = \lambda x'' + (1-\lambda)x \; \text{ with } \; \lambda = \frac{|x'-x|}{|x'' -x|} \in (0,1)\cap \Q \quad \text{and} \quad \frac{2r}{3} \leq |x'' -\bar{x}| \leq r.
\]
By the $\Q$-concavity of $\p$ we obtain
\[
\p(x) -\p(x') = \p(x) - \p(\lambda x'' + (1-\lambda)x) \leq \lambda\left(\p(x) - \p(x'') \right) = |x'-x|\frac{\p(x)-\p(x'')}{|x''-x|}.
\]
Since $x, x'' \in B(\bar{x},r)$, from part \ref{parteuno} we have that $|\p(x)-\p(x'')| \leq 2C$. Furthermore, we have
\[
|x''-x|\geq |x''-\bar{x}|-|x-\bar{x}| \geq \frac{2r}{3}-\frac{r}{2}= \frac{r}{6}, 
\]
which gives 
\[
|\p(x)-\p(x')|\leq \frac{12C}{r}|x-x'|, 
\]
as claimed in part \ref{partedos} of the lemma.
\end{proof} 
\begin{Def}
A function $h \colon N_{\R} \to \R$ is said to be \emph{rational} if it satisfies that $h\left(N_{\Q}\right) \subseteq \Q$.
\end{Def}
The following is one of the main results of this section. 
\begin{theorem}\label{volume}
Let $\D$ be a nef toric $b$-divisor and let $\tilde{\phi}$ be its corresponding $\Q$-concave, conical function (see Remark \ref{remark}). Then there existes a unique continuous, rational and concave function $\phi\colon N_{\R} \to \R$ extending $\tilde{\phi}$. Moreover, $\D$ is integrable and its degree equals $n!$ times the volume of the stability set of $\phi$, i.e
\[
\D^n = n! \, \vol \left(\Delta_{\phi}\right).
\]
\end{theorem}
\begin{proof} Let $\D$ and $\tilde{\phi}$ be as in the statement of the theorem. Consider an element $v \in N_{\R}$ and let $\{v_i\}_{i \in \N}$ be a sequence in $N_{\Q}$ converging to $v$. We define the extension $\phi \colon N_{\R} \to \R$ by setting 
\[
\phi(v) \coloneqq \lim_{i} \p(v_i).
\]
This is well defined by Lemma \ref{lemma1}. Moreover, we claim that for every $v \in N_{\R}$ we have that 
\[
\lim_{\Sigma' \in R(\Sigma)}\psi_{\Sigma'}(v) = \phi(v).
\]
In particular, this means that for all $v \in N_{\R}$, the sequence $\left\{\psi_{\Sigma'}(v)\right\}_{\Sigma' \in R(\Sigma)}$ is Cauchy.
To prove the claim, note that if $v \in N_{\Q}$, then $\phi(v) = \p(v) = \psi_{\tilde{\Sigma}}(v)$ for some $\tilde{\Sigma} \in R(\Sigma)$. Thus, we have that 
\[
\lim_{\Sigma' \in R(\Sigma)}\psi_{\Sigma'}(v) = \psi_{\tilde{\Sigma}}(v) = \p(v) = \phi(v).
\]
If $v \notin N_{\Q}$, by continuity, we get 
\[
 \lim_{\Sigma' \in R(\Sigma)}\psi_{\Sigma'}(v) = \lim_i \lim_{\Sigma' \in R(\Sigma)}\psi_{\Sigma'}(v_i)= \lim_i \p(v_i) = \phi(v),
 \]
 thus proving the claim. Now, let $\Sigma'' \geq \Sigma' \in R(\Sigma)$. Recalling the definition of the polytopes $P_{D_{\Sigma'}}$ and $P_{D_{\Sigma''}}$ (see Section \ref{notations}), we have the following description of the complement $P_{D_{\Sigma'}} \setminus P_{D_{\Sigma''}}$:
\begin{align*}
P_{D_{\Sigma'}} \setminus P_{D_{\Sigma''}} &= \left\{m \in M_{\R}\, \big{|}\, m \in P_{D_{\Sigma'}}\, , \,m \notin P_{D_{\Sigma''}} \right\}\\ 
&= P_{D_{\Sigma'}}\cap \left\{m \in M_{\R} \, \big{|} \exists \tau'' \in \Sigma''(1) \setminus \Sigma'(1) \colon \langle m, v_{\tau''}\rangle < \psi_{\Sigma''}(v_{\tau''})\right\} \\
&= P_{D_{\Sigma'}}\cap \left\{m \in M_{\R} \, \big{|} \exists \tau'' \in \Sigma''(1) \setminus \Sigma'(1) \colon \psi_{\Sigma'}(v_{\tau''}) \leq \langle m, v_{\tau''}\rangle < \psi_{\Sigma''}(v_{\tau''})\right\}.
\end{align*}  
Hence, since the sequence $\left\{\psi_{\Sigma'}(v_{\tau''})\right\}_{\Sigma' \in R(\Sigma)}$ is Cauchy, we have that there exists a refinement $\tilde{\Sigma} \in R(\Sigma)$ such that for all $\Sigma', \Sigma'' \geq \tilde{\Sigma} \in R(\Sigma)$ and for any $\varepsilon > 0$, the inequality 
\[
\vol\left(P_{D_{\Sigma'}} \setminus P_{D_{\Sigma''}}\right)  \leq \varepsilon
\]
is satisfied. Therefore, by the monotonicity property in Lemma \ref{mul}, we have that 
\[
\lim_{\Sigma' \in R(\Sigma)}\vol\left(P_{D_{\Sigma'}}\right) = \vol\left(\Delta_{\phi}\right).
\]
Finally, the limit $\D^n$ exists and is given by 
\[
 \D^n = \lim_{\Sigma' \in R\left(\Sigma\right)}D_{\Sigma'}^n = n! \,\lim_{\Sigma' \in R\left(\Sigma\right)}\vol\left(P_{D_{\Sigma'}}\right) = n! \, \vol\left(\Delta_{\phi} \right),
\]
concluding the proof of the theorem. 
\end{proof}

\begin{cor}\label{brunn_mink}
Let $\D_1$ and $\D_2$ be nef toric $b$-divisors and suppose that the induced convex sets $\Delta_{D_1}$ and $\Delta_{D_2}$ from Theorem \ref{volume} are full-dimensional. Then, we have a Brunn--Minkowski-type inequality for their degree, i.e.~the inequality
\[
(\D_1^n)^{1/n} + (\D_2^n)^{1/n} \leq ((\D_1 + \D_2)^n)^{1/n}
\]
is satisfied.
\end{cor}
\begin{proof} This follows from Theorem \ref{volume} and a standard fact in convex geometry about volumes of convex sets (see e.g.~\cite{Gardner}). 
\end{proof}
Our next goal is to show that the mixed degree of a collection of nef toric $b$-divisors $\D_1, \dotsc, \D_n$ exists. We start with the following lemma:   
\begin{lemma}\label{variables}
For variables $X_1, \dotsc, X_n$ the following holds: 
\[
n! \, X_1 \dotsm X_n = \sum_{I\subseteq \{1, \dotsc, n \} }(-1)^{n- \# I}\left(\sum_{i \in I}X_i\right)^n.
\]
\end{lemma}
\begin{proof}
This follows by induction on $n$.
\end{proof}
\begin{theorem}\label{amelie}
Let $\D_1, \dotsc, \D_n$ be a collection of nef toric $b$-divisors. Then, the mixed degree 
\[
\lim_{\Sigma' \in R\left(\Sigma\right)} D_{1, \Sigma'} \dotsm D_{n, \Sigma'}
\]
exists, is finite, and is given by the mixed volume of the stability sets of the corresponding concave functions $\phi_1, \dotsc, \phi_d$, i.e.~the equality
\[
D_1 \dotsm D_n = \text{MV}\left(\Delta_{\phi_1}, \dotsc, \Delta_{\phi_n}\right)
\] 
holds true. Thus we have a well defined, multilinear map from the set of $n$-tuples of nef toric $b$-divisors to the non-negative real numbers which sends such an $n$-tuple to its mixed degree. 
\end{theorem}
\begin{proof}
We show that the sequence $\left\{D_{1, \Sigma'} \dotsm D_{n, \Sigma'}\right\}_{\Sigma' \in R(\Sigma)}$ is Cauchy. Let $\varepsilon > 0$. For all subsets $I \subseteq \{1, \dotsc, n\}$, we let $\D_{I} \coloneqq \sum_{i \in I}\D_{i}$. Note that $\D_I$ is a positive linear combination of nef toric $b$-divisors and is thus nef. It follows from Theorem \ref{volume} that $\D_{I}$ is integrable. In particular, this means that the sequence $\left\{D_{I,\Sigma'}^n\right\}_{\Sigma' \in R(\Sigma)}$ is Cauchy. Hence, there exists a fan $\tilde{\Sigma}_I \in R(\Sigma)$ such that for all refinements $\Sigma', \Sigma'' \geq \tilde{\Sigma}_I \in R(\Sigma)$, the inequality
\[
\left|D_{I, \Sigma''}^n -D_{I, \Sigma'}^n\right| \leq \frac{n!\varepsilon}{2^n}
\]
is satisfied. Let $\tilde{\Sigma} \coloneqq \max_{I\subseteq \{1, \dots, n\}}\tilde{\Sigma}_I$. Then, using Lemma \ref{variables}, we get that for all $\Sigma', \Sigma'' \geq \tilde{\Sigma}$, the sequence of inequalities     
\begin{align*}
& \left| D_{1, \Sigma''} \dotsm D_{n, \Sigma''}-D_{1, \Sigma'} \dotsm D_{n, \Sigma'}\right| \\ &= \left| \frac{1}{n!}\sum_{I \subseteq \{1, \dotsc, n\}}(-1)^{n-\#I}\left(\sum_{i \in I}D_{i,\Sigma''}\right)^n- \frac{1}{n!}\sum_{I \subseteq \{1, \dotsc, n\}}(-1)^{n-\#I}\left( \sum_{i \in I}D_{i,\Sigma'}\right)^n \right|\\ &=\frac{1}{n!}\left|\sum_{I \subseteq \{1, \dotsc ,n\}}(-1)^{n-\#I}\left[\left(\sum_{i \in I}D_{i,\Sigma''}\right)^n-\left(\sum_{i \in I}D_{i,\Sigma'}\right)^n\right]\right| \\ &= \frac{1}{n!}\left|\sum_{I \subseteq \{1, \dotsc ,n\}}(-1)^{n-\#I}\left[D_{I,\Sigma''}^n-D_{I,\Sigma'}^n\right]\right|\\
&\leq \frac{1}{n!}\sum_{I \subseteq \{1, \dotsc, n \}} \left| D_{I, \Sigma''}^n- D_{I, \Sigma'}^n\right|
\leq \varepsilon
\end{align*}
holds true, as we wanted to show. Moreover, we have that
\begin{align*}
\D_1 \dotsm \D_n &= \lim_{\Sigma' \in R(\Sigma)}D_{1, \Sigma'}\dotsm D_{n, \Sigma'} \\
&= \lim_{\Sigma' \in R\left(\Sigma\right)}\text{MV}\left(P_{D_{1, \Sigma'}}, \dotsc, P_{D_{n, \Sigma'}}\right)\\ &= \text{MV}\left(\Delta_{\phi_1}, \dotsc, \Delta_{\phi_n}\right),
\end{align*}
concluding the proof of the theorem. 
\end{proof}
As a corollary we get that also the difference of nef toric $b$-divisors is integrable.
\begin{theorem}\label{differencenef}
Let $\D = \D_1 - \D_2$ be the difference of two nef toric $b$-divisors $\D_1$ and $\D_2$. Then $\D$ is integrable and its degree is given as a sum of mixed volumes of the stability sets of the concave functions $\phi_1$ and $\phi_2$ corresponding to $\D_1$ and $\D_2$, respectively. Explicitely, we have 
\[
\D^n = \sum_{i =0}^n (-1)^i \binom{n}{i} \MV\left(\underbrace{\Delta_{\phi_1}, \dotsc, \Delta_{\phi_1}}_{(n-i)\timess}, \underbrace{\Delta_{\phi_2}, \dotsc, \Delta_{\phi_2}}_{i\timess}\right).
\]
\end{theorem}
\begin{proof}
We show that the sequence $\left\{D_{\Sigma'}^n\right\}_{\Sigma' \in R(\Sigma)} = \left\{\left(D_{1, \Sigma'}-D_{2, \Sigma'}\right)^n\right\}_{\Sigma' \in R(\Sigma)}$ is Cauchy. Let $\varepsilon > 0$. By the proof of Theorem \ref{amelie}, the nefness of $\D_1$ and $\D_2$ implies that for all $i = 1, \dotsc, n$ there exists a fan $\tilde{\Sigma}_i \in R(\Sigma)$ such that for all refinements $\Sigma', \Sigma'' \geq \tilde{\Sigma}_i$, the inequality 
\[
\left|D_{1, \Sigma''}^{n-i}D_{2, \Sigma''}^i - D_{1, \Sigma'}^{n-i}D_{2, \Sigma'}^i\right| \leq \frac{\varepsilon}{(n+1)\binom{n}{i}}
\]
holds true. Letting $\tilde{\Sigma} \coloneqq \max_{i \in \{1, \dotsc, n\}}\tilde{\Sigma}_i$, we get   
\begin{align*} 
\left|D_{\Sigma''}^n - D_{\Sigma'}^n\right| =& \left|(D_{1,\Sigma''} -D_{2, \Sigma''})^n - (D_{1,\Sigma'} -D_{2, \Sigma'})^n\right| \\=& \left|\sum_{i =0}^n (-1)^i \binom{n}{i}\left(D_{1, \Sigma''}^{n-i}D_{2, \Sigma''}^i - D_{1, \Sigma'}^{n-i}D_{2, \Sigma'}^i \right)\right| \\
\leq & \sum_{i=0}^n \binom{n}{i}\left|D_{1, \Sigma''}^{n-i}D_{2, \Sigma''}^i- D_{1, \Sigma'}^{n-i}D_{2, \Sigma'}^i\right|\leq \varepsilon
\end{align*}
for all $\Sigma', \Sigma'' \geq \tilde{\Sigma}$, as claimed. Moreover, we get 
\begin{align*}
\D^n &= \lim_{\Sigma' \in R(\Sigma)}D_{\Sigma'}^n \\
&= \lim_{\Sigma' \in R(\Sigma)}\left(D_{1, \Sigma'} - D_{2, \Sigma'}\right)^n\\
&= \lim_{\Sigma' \in R(\Sigma)}\sum_{i =0}^n (-1)^i \binom{n}{i}D_{1, \Sigma'}^{n-i}D_{2, \Sigma'}^i \\
&= \lim_{\Sigma' \in R(\Sigma)}\sum_{i =0}^n (-1)^i \binom{n}{i} \MV\left(\underbrace{P_{D_{1,\Sigma'}}, \dotsc, P_{D_{1, \Sigma'}}}_{(n-i)\timess}, \underbrace{P_{D_{2,\Sigma'}}, \dotsc, P_{D_{2, \Sigma'}}}_{i\timess}\right) \\
&= \sum_{i =0}^n (-1)^i \binom{n}{i} \MV\left(\underbrace{\Delta_{\phi_1}, \dotsc, \Delta_{\phi_1}}_{(n-i)\timess}, \underbrace{\Delta_{\phi_2}, \dotsc, \Delta_{\phi_2}}_{i\timess}\right)
\end{align*}
as stated in the theorem. 
\end{proof}
\section{The case $n=2$}\label{2dimensionalcase}\label{section4}
In this section, we consider the case of smooth and complete toric surfaces, i.e.~$\Sigma \subseteq N_{\R}$ is a smooth and complete fan of dimension $2$. We see that integrability of a toric $b$-divisor on a toric surface is equivalent to the convergence of a certain series, where the sum ranges over pairs of relatively prime integers. We also provide some examples.
 
We start by fixing an identification of lattices $N \simeq \Z^2$ and we consider a primitive vector $v = (v_1, v_2) \in N^{\text{prim}}.$ We denote by $P_v$ the Newton polyhedron associated to the monomial ideal  $I= (x^{v_2},y^{v_1})$. Note that $P_v$ is not bounded. Also, we let $\Sigma_{P_v} \subseteq N_{\R}$ be the (non complete) normal fan of $P_v$. Figure \ref{polyandfan} is an example of the polyhedron $P_v$ and of its normal fan $\Sigma_{P_v}$ for the primitive vector $v = (2,3) \in \Z^2$. The fan $\Sigma_{P_v}$ consists of the two $2$-dimensional cones $\sigma_1$ and $\sigma_2$, the three rays spanned by the vectors $(1,0), (0,1)$ and $(2,3)$ and the $0$-dimensional cone $(0,0)$.  

\begin{figure}[ht!!!]
\begin{center}
\begin{tikzpicture}[scale = 1.3]
    \draw [<->,thick] (0,2) node (yaxis) [above] {$y$}
        |- (3,0) node (xaxis) [right] {$x$};
    \draw[-,blue, very thick] (0,1) -- (1.5,0) ;
   \draw[-,blue, very thick] (0,1) -- (0,2);
    \draw[-,blue, very thick] (1.5,0) -- (3,0);
    \draw[blue, very thick] (1.5,1.1)node{$P_v$};
    \draw (1.5,-0.2) node{$(3,0)$};
    \draw (-0.4,1) node{$(0,2)$};

\filldraw[blue, dashed, opacity=0.1] (1.5,0)--(3,0)--(3,2)--(0,2)--(0,1)--cycle;
    \draw [<->,thick] (5,2) node (yaxis) [above] {$y$}
        |- (8,0) node (xaxis) [right] {$x$};
    \draw[->,blue, very thick] (5,0) -- (5,0.5) ;
   \draw[->,blue, very thick] (5,0) -- (5.5,0);
    \draw[->,blue, very thick] (5,0) -- (6,1.5) node[right]{$(2,3)$};
    \draw[blue, very thick] (8.2,2.2)node{$\Sigma_{P_v}$};
    \draw[blue, very thick] (4.6, 0.5)node{$(0,1)$}; 
    \draw[blue, very thick] (5.5, -0.2)node{$(1,0)$}; 
   \draw[blue, very thick] (6.5, 0.4)node{$\sigma_1$};
   \draw[blue, very thick] (5.3, 1.1)node{$\sigma_2$};
\end{tikzpicture}
\end{center}
\caption{Picture of $P_{(2,3)}$ and its fan $\Sigma_{P_{(2,3)}}$}
\label{polyandfan}
\end{figure}
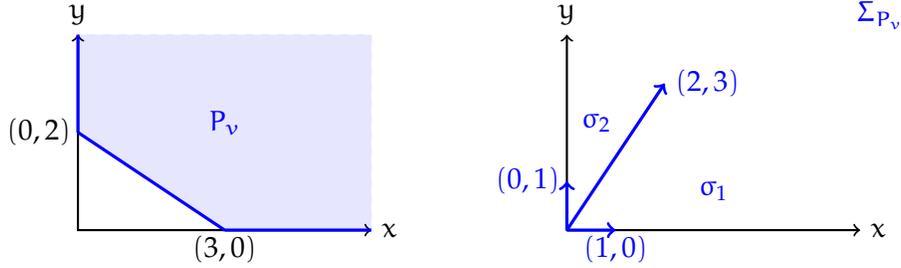
Even though the fan $\Sigma_{P_v}$ is not complete, it nevertheless corresponds to a smooth (non complete) toric variety $X_{\Sigma_{P_v}}$ and in the same manner as in the complete case, a regular subdivision $\Sigma'_{P_v} \geq \Sigma_{P_v}$ corresponds to a smooth (non complete) toric variety $X_{\Sigma'_{P_v}}$ together with a proper toric birational morphism $X_{\Sigma'_{P_v}} \to X_{\Sigma_{P_v}}.$ 
We have the following lemma. 
\begin{lemma}
With the same notations as above, there exists a minimal smooth subdivision $\Sigma'_{P_v}$ of $\Sigma_{P_v}$. Here, by minimal smooth subdivision we mean a smooth subdivision $\Sigma'_{P_v}$ of $\Sigma_{P_v}$ such that the induced toric morphism $X_{\Sigma'_{P_v}}\to X_{\Sigma_{P_v}}$ gives a \emph{minimal resolution of singularities} as defined in \cite[Section~10]{CLS}.  
\end{lemma}
\begin{proof}
See \cite[Corollary~10.4.9]{CLS}.
\end{proof}
\begin{rem}\label{amelie1}
We want to give an explicit description of the fan $\Sigma'_{P_v}$ in the above theorem. Or at least, of the rays $\tau_{v_{\alpha}}$ and $\tau_{v_{\beta}}$ which together with $\tau_v$ span a $2$-dimensional cone in $\Sigma'_{P_v}$, i.e.~such that $\text{cone}\langle v, v_{\alpha}\rangle$, $\text{cone}\langle v, v_{\beta}\rangle \in \Sigma'_{P_v}(2)$. The following lemma, whose proof follows from the computations done in the proof of \cite[Theorem~10.2.3]{CLS} gives a formula for the vectors $v_{\alpha}$ and $v_{\beta}$. 
\end{rem} 
\begin{lemma}\label{euc}
Let $v = (v_1, v_2) \in N^{\prim}$. Let $x$ and $y$ be integers uniquely defined by the following properties:
\begin{enumerate}
\item $y \leq 0$ and $x \geq 0$.
\item $xv_1 + yv_2 = 1$.
\item $0 \leq -y < v_1$.
\item $0\leq x < v_2$.
\end{enumerate}
Then the primitive vectors $v_{\alpha}$ and $v_{\beta}$ of Remark \ref{amelie1} are given by 
\begin{align}\label{alphabeta}
v_{\alpha} = (-y,x) \quad \text{ and }\quad v_{\beta} = (v_1+y,v_2-x).
\end{align}
\end{lemma}
\begin{rem}
The construction of the polyhedral subdivision $\Sigma'_{P_v}$ which can be found in the proof of \cite[Corollary~10.4.9/Theorem 10.2.3]{CLS} is done by taking successive barycentric subdivisions of the $2$-dimensional cones of $\Sigma_{P_v}$. It thus follows that the ray $\tau_v$ is obtained from the barycentric subdivision of $\cone(v_{\alpha}, v_{\beta})$. Indeed, this corresponds to the fact that $v = v_{\alpha} + v_{\beta}$, which follows from a straightforward computation using the formulas in (\ref{alphabeta}).
\end{rem}
Figure \ref{starv} shows the vectors $v_{\alpha}$ and  $v_{\beta}$ for the case $v = (2,3)$. Here we have $v_{\alpha} = (1,2)$ and $v_{\beta} = (1,1)$.

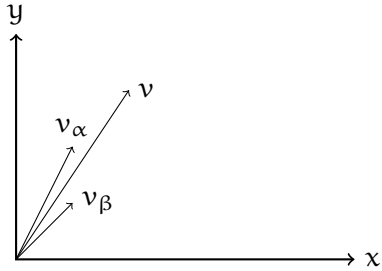
\begin{figure}[ht!]
\begin{center}
\begin{tikzpicture}[scale=1.5]
    \draw [<->,thick] (0,2) node (yaxis) [above] {$y$}
        |- (3,0) node (xaxis) [right] {$x$};
    \draw [->] (0,0) -- (1,1.5) node [right]{$v$};
    \draw [->] (0,0) -- (0.5,1) node [above]{$v_{\alpha}$};
    \draw [->] (0,0) -- (0.5,0.5) node [right]{$v_{\beta}$};
    
\end{tikzpicture}
\end{center}
\caption{Picture for the case $v = (2,3), \, v_{\alpha} = (1,2), \,v_{\beta} = (1,1)$}
\label{starv}
\end{figure}
Let's return to toric $b$-divisors. Let $\D = (D_{\Sigma'})_{\Sigma' \in R\left(\Sigma\right)}$ be a toric $b$-divisor on a smooth and complete toric surface $\X$. We define a function 
\[
\mu_{\D} \colon N^{\text{prim}} \longrightarrow \Q
\]
in the following way: let $v= (v_1, v_2) \in N^{\text{prim}}$. If $\tau_v \in \Sigma(1)$ then we put $\mu_{\D}(v) = 0$. Otherwise, let $\sigma \in \Sigma(2)$ be a $2$-dimensional cone in $\Sigma$ containing $\tau_v$ in its interior. Since $\Sigma$ is smooth, we may assume that $\sigma = \cone(e_1, e_2 )$ is the positive quadrant. Consider the fan $\Sigma'_v = \Sigma \cup \Sigma'_{P_v}$, where $\Sigma'_{P_v}$ is the normal fan of the Newton polyhedron $P_v$ described previously. By construction, we have that $\Sigma'_v$ is a smooth subdivision of $\Sigma$, i.e.~it belongs to $R\left(\Sigma\right)$. Consider the toric birational morphism $\pi \colon X_{\Sigma'_v} \to X_{\Sigma'_v \setminus{\{\tau_v\}}}$. This is the blow up of $X_{\Sigma'_v \setminus{\{\tau_v\}}}$ at the torus fixed point corresponding to $\text{cone}\langle v_{\alpha}, v_{\beta} \rangle \in \Sigma'_v \setminus{\{\tau_v\}}(2)$ with exceptional divisor $E_{\tau_v}$ corresponding to the ray $\tau_v$. Define $\mu_{\D}(v) \in \Q$ to be the rational number satisfying 
\[
D_{\Sigma'_v} = \pi^*D_{\Sigma'_v \setminus{\{\tau_v\}}} + \mu_{\D}(v)E_{\tau_v}.
\]
We call $\mu_{\D}$ the \emph{jumping function} of the toric $b$-divisor $\D$.  
\begin{lemma}\label{madrid}
Let $\p$ be the conical function associated to the $b$-divisor $\D$ from Lemma \ref{equi}. Then the function $\mu_{\D} \colon N^{\prim} \to \Q$ is given by the assignment
\[
v \longmapsto \p(v) - (\p(v_{\alpha}) + \p(v_{\beta})),
\]
where $v_{\alpha}$ and $v_{\beta}$ are the primitive vectors of Lemma \ref{euc}.
\end{lemma} 
\begin{proof}
Fix a primitive vector $v \in N^{\text{prim}}$ and let $\sigma = \text{cone}\langle v_{\alpha}, v_{\beta} \rangle$. Let $\psi_{\sigma}$ be the linear function on $\sigma$ determined by the values $\p\left(v_{\alpha}\right)$ and $\p\left(v_{\beta}\right)$. Then, the coefficient $\mu_{\D}(v)$ we are looking for is the difference $\p(v) - \psi_{\sigma}(v)$. We compute 
\[
\p(v) - \psi_{\sigma}(v) = \p(v)-\left(\psi_{\sigma}\left(v_{\alpha}\right)+\psi_{\sigma}\left(v_{\beta}\right)\right) = \p(v) - (\p(v_{\alpha}) + \p(v_{\beta}))
\] 
as claimed.
\end{proof}

By definition, integrability of a toric $b$-divisor $\D$ is equivalent to the convergence of the sum $\sum_i\left(D_{\Sigma_{i+1}}^2-\pi_{i+1}^*D_{\Sigma_{i}}^2\right)$ of differences of toric degrees for every chain of blow ups 
\[
\dotsm \xrightarrow[]{\pi_{n+1}} X_{\Sigma_n} \xrightarrow[]{\; \; \pi_n  \; \;} \dotsm \xrightarrow[]{\; \; \pi_1 \; \;} X_{\Sigma_0} = \X.
\]
To simplify notation, let us write $a_i \coloneqq D_{\Sigma_{i+1}}^2-\pi_{i+1}^*D_{\Sigma_{i}}^2$. Note that we have canonical chains of blow ups consisting of taking successive barycentric subdivisions corresponding to blow ups of torus fixed points. Moreover, by Remark \ref{starcofinite}, the refinements $\Sigma' \geq \Sigma$ corresponding to barycentric subdivisions are cofinal in $R(\Sigma)$. Also, by the intersection theory of toric varieties we have $(E_{\tau_v})^2 =-1$ (see e.g. \cite[Example~10.4.5]{CLS}). Hence, replacing the index set $\{i \in \N\}$ with the index set $\{v \in N^{\prim}\}$, we have that $a_v = -\mu(v)^2$. The following theorem follows from these considerations together with Lemma \ref{madrid}.
\begin{theorem}\label{f}
Let notations be as above. A toric $b$-divisor $\D$ on a toric surface is integrable if and only if the infinite sum
\begin{align*}
\sum_{v \in N^{\text{prim}}}\mu_{\D}(v)^2 = \sum_{\overset{v=(v_1,v_2) \in \Z^2}{\gcd(v_1,v_2) = 1}}\big{(}\p(v) - \p(v_{\alpha}) - \p(v_{\beta})\big{)}^2
\end{align*}
is a finite number.
Moreover, if this is the case, and if we denote the value of the series by $\mu$, then the equality 
\[
\D^2 = D_{\Sigma}^2 -\mu
\]
holds true. 
\end{theorem} 
Let's turn to an example. 
\begin{exa}\label{exa1}
We choose and identification of lattices $N\simeq \Z^2$ and consider the fan $\Sigma \subseteq \R^2$ corresponding to the projective plane $\P^2$. Let $\D$ be the nef toric $b$-divisor given by the concave function 
\[
\phi \colon \R^2 \longrightarrow  \R
\]
given by the assignment
\[
(a,b) \longmapsto  \begin{cases} \frac{ab}{a+b} \hspace{1.8cm} &\text{if \hspace{0.1cm} } a,b \geq 0, \, \text{ and } \, a+b >0,\\ \min\{a,b\} &\text{otherwise}. \end{cases} 
\]
Let $v=(v_1, v_2)$ be a pair of positive integers with $\text{gcd}(v_1, v_2) =1$ and let $x,y$ be integers satisfying the conditions in Lemma \ref{euc}. Therefore, we have 
\[
v_{\alpha} = (-y, x) \quad \text{ and } \quad v_{\beta} = (v_1+y, v_2 - x).
\]
We compute 
\begin{align*}
\mu_{\D}(v) &= \phi(v) - \phi(v_{\alpha}) - \phi(v_{\beta}) \\ &= \frac{v_1v_2}{v_1 + v_2} + \frac{xy}{x-y}-\frac{(v_1+y)(v_2-x)}{v_1+v_2-x+y} \\ &= \frac{(v_1x + v_2y)^2}{(v_1 + v_2)(x-y)(v_1+v_2-x+y)} \\&= \frac{1}{(v_1 + v_2)(x-y)(v_1+v_2-x+y)}.
\end{align*}
Taking the sum over all coprime integers we get 
\begin{align*}
\mu &= \sum_{\overset{v=(v_1,v_2) \in \Z^2}{\gcd(v_1,v_2) = 1}} \mu\left(v\right)^2 =  \sum_{\overset{v=(v_1,v_2) \in \Z^2}{\gcd(v_1,v_2) = 1}}\left(\frac{1}{(v_1 + v_2)(x-y)(v_1+v_2-x+y)}\right)^2 \\ &= \sum_{\overset{(m,n) \in \Z^2}{\gcd(m,n) = 1}}\frac{1}{m^2n^2(m+n)^2} = \frac{1}{3}.
\end{align*}
For the third equality, we set $m=v_1 + v_2$ and $n= x-y$ (note that $xm + v_2n = 1$). For the fourth equality we refer to \cite[Section~4]{BKK} and to \cite[Section~3]{AK}.
Combining the above calculation with Theorem \ref{f} we get
\[
\D^2 = D_{\P^2}^2 - \frac{1}{3} = H^2 -\frac{1}{3}= 1 - \frac{1}{3} = \frac{2}{3},
\]
where $H$ is the class of a hyperplane in $\P^2$. 

Note that since $\D$ is nef, we can compute this degree also using the formula in Proposition \ref{volume}. Now, to compute the stability set $\Delta_{\phi}$ we proceed as follows: first, let us recall that $\Delta_{\phi}$ is defined by 
\[
\Delta_{\phi} = \left\{(x,y) \in \left(\R^2\right)^{\vee}\, |\,  xa + yb  -\phi(a,b) \text { is bounded below }\forall (a,b) \in \R^2 \right\}
\]
(see Section \ref{notations}).
The function $\phi$ is concave and it has two regions where it is linear. In these regions, the slope is given by $(1,0)$ and $(0,1)$. This gives us the vertices of $\Delta_{\phi}$ and the line $(1,0) + t(0,1)$ for $t \in [0,1]$. In order to compute the non-linear part, we have to look at the tangent lines at $\phi$ in the positive orthant. We compute 
\[
\phi_a \coloneqq \frac{\partial \phi}{\partial a} = \frac{b^2}{(a+b)^2} \quad \text{ and }\quad \phi_b\coloneqq\frac{\partial \phi}{\partial b} = \frac{a^2}{(a+b)^2}.
\]
These are homogeneous functions of degree $0$ and hence we can see them as functions of $t = a/b$. The points corresponding to the curved boundary of the convex set $\Delta_{\phi}$ have coordinates given by $ (x,y) = \left((\phi_a(t), \phi_b(t)\right)$ with $t \in (0,\infty)$. One sees that these points satisfy  
\[
\sqrt{x} + \sqrt{y} = 1.
\]
Hence, we get the following description of the stability set of $\phi$:
\[
\Delta_{\phi} = \big{\{} (x,y) \in \R^2 \,\big{|} \, x,y \geq 0, \quad x+y \leq 1 , \quad \sqrt{x} + \sqrt{y} \geq 1 \big{\}}.
\]
Finally, letting $\Delta$ be the standard simplex $\convhull\big{(}(0,0), (1,0), (0,1) \big{)}$, we compute
\[
\D^2 = 2 ! \,\vol \left(\Delta_{\phi}\right) = 2 \, \vol\left(\Delta \right) - 2 \int_0^1 \left(1-\sqrt{x}\right)^2 \, \text{d}x = 1 - \frac{1}{3} = \frac{2}{3}.
\] 
Figure \ref{cuttingsimplices} relates the previous two calculations. Note that the $\mu_{\D}$-values (which are the numbers between the parenthesis depicted on the right) correspond to the volumes of the simplices which we successively subtract from the standard simplex $\Delta$ in order to approximate the volume of the convex set $\Delta_{\phi}$.
\begin{rem}
As is mentioned in the articles \cite{BKK} and \cite{AK}, the number $\mu$ can be interpreted as the value of the \emph{Mordell--Tornheim zeta function} evaluated at $(2,2;2)$. The latter paper gives a deeper insight into the connection between multiple zeta functions and arithmetic intersection numbers via volumes of simplices. This is a connection which should be further investigated in the future. In particular, intersection numbers of toric $b$-divisors in toric varieties of any dimension should be interpreted as multiple zeta values.
\end{rem}
\end{exa}
Let $\D$ be a toric $b$-divisor and let $\p$ be its associated function on $N_{\Q}$. A natural question to ask is what are necessary and sufficient conditions on $\p$ which ensure the integrability of $\D$. If $\p$ does not extend to a continuous function $\phi$ on $N_{\R}$ there is not much we can say. The interesting case is if we assume that the extension exists. In this case, what properties must $\phi$ satisfy in order to ensure integrability? We have already seen that if $\phi$ is a difference of concave functions, then the $b$-divisor is integrable. But can we ask for less? Is the existence of an extension enough? 

The answer to the last question is negative, as can be shown by Example \ref{divergence}. Before giving it, we introduce some notation.
\begin{Def}
Let $f,g \colon \N \to \R$ be two $\R$-valued functions with domain the natural numbers $\N$. We say that $f$ is equivalent to $g$ (and write $f \sim g$), if there exist two positive real constants $B, C$ such that $Bf(n) \leq g(n) \leq Cf(n)$ for all $n\gg0$.
\end{Def}
Now, we give an example of a non-integrable toric $b$-divisor for which the associated function $\p$ on $\N_{\Q}$ extends to $\N_{\R}$. Note that this is actually a toric $\R$-$b$-divisor, not a toric $\Q$-$b$-divisor. The question whether we can find a non-integrable toric $\Q$-$b$-divisor whose corresponding function $\tilde{\phi}$ on $N_{\Q}$ extends continuously to $N_{\R}$ still remains open. 
 \begin{center}
  \begin{figure}
\includegraphics[scale=0.14]{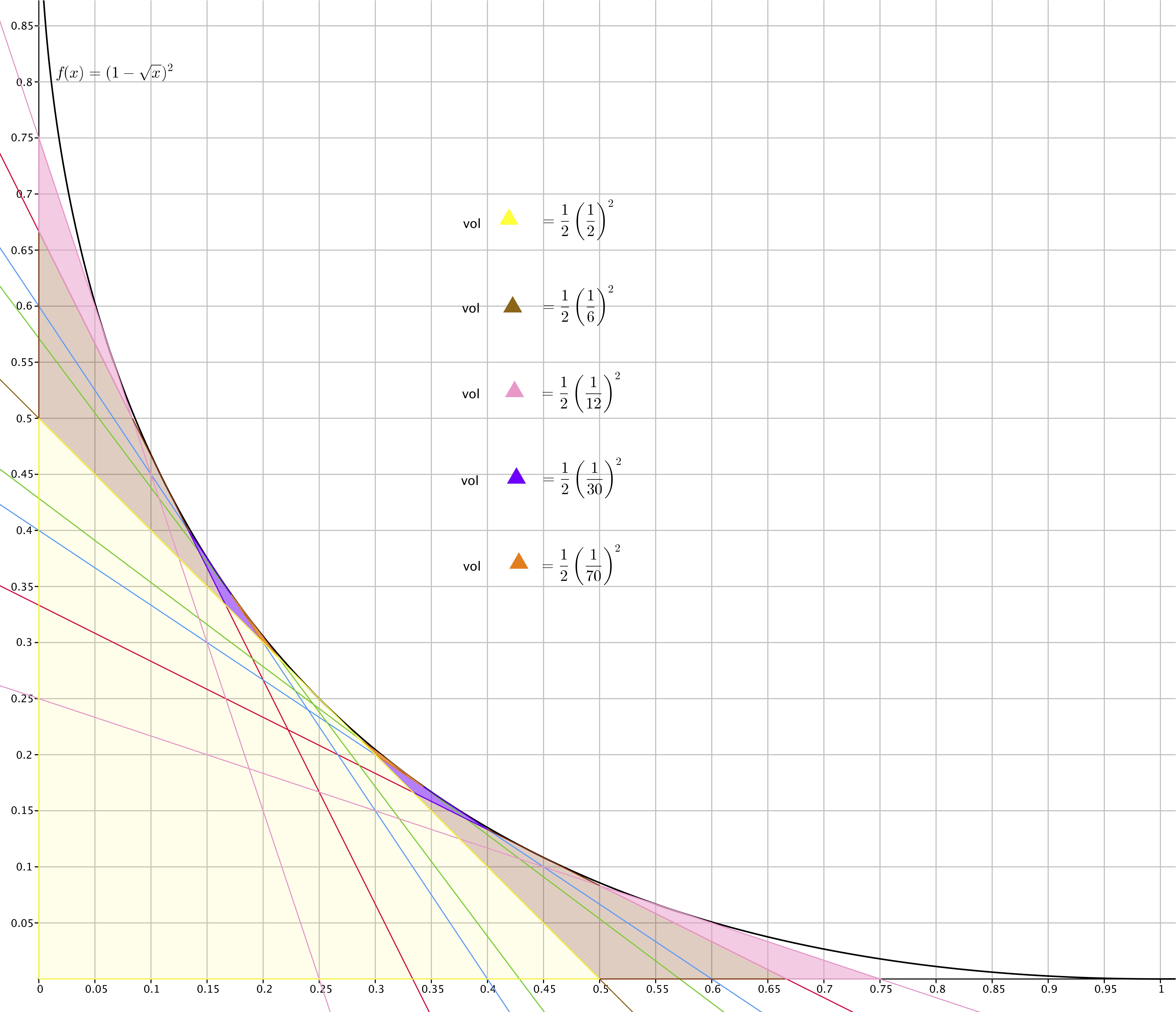}
\caption{Cutting of simplices and $\mu$-values}\label{cuttingsimplices}
\end{figure}
\end{center}
\begin{exa}\label{divergence}
Let $\X = \P^2$. We choose an identification of lattices $N \simeq \Z^2$ and consider the toric $b$-divisor given by the conical function
\[
\phi \colon \R^2 \longrightarrow  \R
\]
defined by 
\[
(a,b) \longmapsto  \begin{cases} \sqrt{|a||b|}  \hspace{1.5cm} &\text{if \hspace{0.1cm} } a,b \geq 0, \\ \min\{a,b\} &\text{otherwise}. \end{cases} 
\]

Note that the graph of this function mimics a cusp at the origin. Now, let $n$ be any positive integer and consider the primitive vector $v = (n,1)$. We have $v_{\alpha} = (n-1,1)$ and $v_{\beta} = (1,0)$.
\begin{center}
\begin{tikzpicture}[scale=1.5]
    \draw [<->,thick] (0,2) node (yaxis) [above] {$y$}
        |- (3,0) node (xaxis) [right] {$x$ \hspace{0.2cm} $v_{\beta} = (1,0)$};
    \draw [->] (0,0) -- (2,0.3) node [right]{$v = (n,1)$};
    \draw [->] (0,0) -- (1,0.3) node [above]{$v_{\alpha} = (n-1,1)$};
\end{tikzpicture}
\end{center}   
Using homogeneity we compute 
\begin{align*}
\mu(v) &= \phi(v) - \phi(v_{\alpha}) - \phi(v_{\beta}) \\&= n\phi\left(1, \frac{1}{n}\right) - (n-1)\phi\left(1, \frac{1}{n-1}\right) - \phi(1,0) \\ &= n \left(\frac{1}{\sqrt{n}} - \frac{1}{\sqrt{n-1}}\right) + \frac{1}{\sqrt{n-1}} \\&\sim \frac{1}{\sqrt{n}} .
\end{align*}
Hence,
\[
\sum_{\overset{v=(v_1,v_2) \in \Z^2}{(v_1,v_2) = 1}}\mu\left(v\right)^2 \gg \sum_{n \geq 0}\frac{1}{n},
\]  
which diverges.
\end{exa} 
\begin{rem} In exactly the same way as in the above example, one can show that a toric $\R$-$b$-divisor on $\P^2$ given by the conical function
\[
\phi \colon \R^2 \longrightarrow \R 
\]
defined by
\[
(a,b) \longmapsto \begin{cases} |a|^{\varepsilon}|b|^{1-\varepsilon} \hspace{1.5cm} & \text{if \hspace{0.1cm} } a,b \geq 0, \\ \min\{a,b\} &\text{otherwise, } \end{cases} 
\]
where $\varepsilon \in (0,1)$, fails to be integrable.  
\end{rem}
We make the following remark. 
\begin{rem}
We have seen in Lemma \ref{differencenef} that the difference of nef toric $b$-divisors is integrable. On the other hand, it is a classical result in algebraic geometry that every divisor on a smooth projective algebraic variety can be written as a difference of very ample divisors (in particular as the difference of nef divisors). The last example shows that this algebro-geometric fact does \emph{not} extend to $b$-divisors. 
\end{rem}

\section{$b$-Convex bodies and global sections of toric $b$-divisors}\label{section5}
In this section, we start by classifying the class of convex bodies which arise from nef toric $b$-divisors. We then proceed to describe the space of global sections of a (not necessarily nef) toric $b$-divisor $\D$ in terms of lattice points in a more general convex set associated to the $b$-divisor, denoted by $\Delta_{\D}$. In the nef case, both convex sets $\Delta_{\D}$ and $\Delta_{\phi_{\D}}$ agree. In this case, we also show that a Hilbert--Samuel type formula holds relating the degree of the nef toric $b$-divisor with the asymptotic growth of the dimension of the space of global sections of multiples of the $b$-divisor. We also define the graded ring of global sections of multiples of a toric $b$-divisor and study the relationship between its finite generation and the polyhedrality of the associated convex set. 

\subsection*{Convex bodies arising as a $b$-convex set}
Recall from Theorem \ref{volume} that to a toric nef $b$-divisor $\D$ one can associate two objects: the concave extension $\phi_{\D} \colon N_{\R} \to \R$ and its stability set $\Delta_{\phi_{\D}} \subseteq M_{\R}$ which is a bounded convex body. A natural question to ask is the following: given a bounded convex set $K$ in $M_{\R}$, does there exist a nef toric $b$-divisor $\D$ such that $K = \Delta_{\phi_{\D}}$? We answer this question in the following proposition. 
 \begin{prop}\label{b-convex} The map given by the assignment
 \[
 \D \longmapsto \Delta_{\phi_{\D}}
 \]
induces a bijective correspondence between the set of nef toric $b$-divisors on $\X$ and the set of bounded convex bodies in $M_{\R}$  whose conical, concave support function $h \colon N_{\R} \to \R$ is rational, i.e.~satisfies that $h\left(N_{\Q}\right) \subseteq \Q$.

Moreover, nef toric $b$-divisors on $\X$ which are Cartier correspond to rational polytopes whose normal fan belongs to $R\left(\Sigma\right)$. These have rational, piecewise linear support functions with regions of linearity given by the normal fan of the rational polytope.    
 \end{prop}
 \begin{proof}
 Let $\D$ be a nef toric $b$-divisor. Then $\phi_{\D}(N_{\Q}) \subseteq \Q$ follows from the proof of Theorem \ref{volume} and the definition of $\phi_{\D}$. Conversely, given a bounded convex body $K$ with rational support function $h_K$ as in the statement of the proposition, Lemma \ref{equi} gives a toric $b$-divisor $\D$. Concavity of $h_{K}$ implies that $\D$ is nef. These maps are clearly inverse to each other.\\
 Now, if $\D$ is moreover Cartier, then $\D$ is determined on some toric model $X_{\Sigma'}$ for some $\Sigma' \in R\left(\Sigma\right)$. Hence, $\phi_{\D} = \phi_{D_{\Sigma'}}$ satisfies $\phi_{D_{\Sigma'}}(N_{\Q}) \subseteq \Q$ and is piecewise linear with linearity regions given by $\Sigma'$ and the stability set $\Delta_{\phi_{\D}} = \Delta_{\phi_{D_{\Sigma'}}}$ is a polytope with vertices in $M_{\Q}$. Conversely, starting with a rational polytope $P$ whose normal fan $\Sigma_P$ is in $R\left(\Sigma\right)$, its corresponding support function is piecewise linear with rational slopes and regions of linearity given by $\Sigma_P$. The corresponding Cartier nef toric $b$-divisor is the one determined by the toric divisor $D_P$ on $X_{\Sigma_P}$ associated to the rational polytope~$P$.       
 \end{proof}
 \begin{Def}\label{b-convexset}
A convex set which corresponds to a nef toric $b$-divisor as above is called a \emph{$b$-convex set}. 
 \end{Def}
\subsection*{Global sections of toric $b$-divisors}
\begin{Def}
Let $D = \sum_{i=1}^ra_iD_i$ be a $\Q$-Weil divisor on an algebraic variety $X$. We define the $\Z$-Weil divisor $\lfloor D \rfloor$ by setting 
\[
\lfloor D \rfloor \coloneqq \sum_{i=1}^r \lfloor a_i \rfloor D_i,
\]
where $\lfloor x \rfloor$ denotes the largest integer less than or equal to a rational number $x$.
\end{Def} 
\begin{Def} let $D$ be a $\Q$-Weil divisor on a smooth algebraic variety $X$. Then its \emph{space of global sections} is defined as 
\[
H^0(X, \O(D)) \coloneqq H^0(X, \O(\lfloor D \rfloor )).
\]
\end{Def}
Now, let $D_{\Sigma}= \sum_{\tau \in \Sigma(1)}a_{\tau}D_{\tau}$ be a (not necessarily nef) toric $\Q$-Weil divisor on $X_{\Sigma}$ and let $P_{D_{\Sigma}}\subseteq M_{\R}$ be its associated rational polyhedron. The proposition below is a classical result in toric geometry. However, since we are going to use similar ideas to generalize this to the $b$-setting, we give the proof also in this classical case. 
\begin{prop}\label{chupo}
With notations as above, the space of global sections of $D_{\Sigma}$ is given by 
\begin{align*}
H^0(X, \O(D_{\Sigma})) = \bigoplus_{m \in P_{D_{\Sigma}}\cap M} k \cdot \chi^m. 
\end{align*} 
\end{prop}
\begin{proof}
In order to simplify notation, we set $D = D_{\Sigma}$. The action of the torus $\T$ on itself given by multiplication induces an action of $\T$ on $k[M]$ as follows: if $t \in \T$ and $f \in k[M]$, then $t \cdot f \in k[M]$ is defined by $p \mapsto f(t^{-1}\cdot p)$. Then, by \cite[Section~4.3]{CLS}, the space $H^0(X_{\Sigma}, \O(\lfloor D \rfloor )) \subseteq k[M]$ is stable under this action and hence, by \cite[Lemma~1.1.16]{CLS}, it is given by 
\begin{align}\label{upsi1}
H^0(X_{\Sigma}, \O(\lfloor D \rfloor )) = \bigoplus_{\chi^m \in H^0(X_{\Sigma}, \O(\lfloor D \rfloor ))}k \cdot \chi^m.
\end{align}
Using that the toric divisor corresponding to a character $\chi^m$ for $m \in M$ is given by 
\[
\text{div}(\chi^m) = \sum_{\tau \in \Sigma(1)}\langle m, v_{\tau}\rangle D_{\tau},
\]
we conclude that 
\begin{align}\label{upsi2}
\chi^m \in H^0(X_{\Sigma}, \O(\lfloor D \rfloor )) \quad &\text{iff} \quad \text{div}(\chi^m) \geq - \lfloor D \rfloor \nonumber \\ 
& \text{iff} \quad \langle m, v_{\tau} \rangle \geq - \lfloor a_{\tau} \rfloor \geq -a_{\tau} \nonumber \\ 
& \text{iff} \quad m \in P_{D_{\Sigma}}.
\end{align}
Combining (\ref{upsi1}) and (\ref{upsi2}), the statement of the proposition follows.
\end{proof}
Now, let $\D= (D_{\Sigma'})_{\Sigma' \in R\left(\Sigma\right)}$ be a toric $b$-divisor which is not necessarily nef. We set 
\[
\lfloor \D \rfloor \coloneqq (\lfloor D_{\Sigma'} \rfloor)_{\Sigma' \in R\left(\Sigma\right)}.
\] Also, note that a rational function $f \in k(\X)^{\times}$ defines a $b$-divisor by setting 
\[b\text{-}\div(f) \coloneqq \left(\div_{\Sigma'}(f)\right)_{\Sigma' \in R\left(\Sigma\right)}.
\]
We now define the space of global sections of a toric $b$-divisor. 
\begin{Def}\label{def:globalsectionstoric}
The \emph{space of global sections} $H^0(X_{\Sigma}, \O(\D))$ of $\D$ is defined by  
\[ H^0(X_{\Sigma}, \O(\D)) \coloneqq \left\{ f \in k(\X) \, \big{|}\, \text{$b$-div}(f) + \lfloor \D \rfloor  \geq 0 \right\} \cup \{0\}.
\]
\end{Def}
The following lemma is a direct consequence of the definitions. 
\begin{lemma}\label{chupito}
The space of global sections of a toric $b$-divisor $\D$ is given as the intersection of the spaces of global sections of the incarnations of $\D$ in the models $X_{\Sigma'}$ as $\Sigma'$ varies in $R\left(\Sigma\right)$, i.e.~the equality
\[
H^0(X_{\Sigma}, \O(\D)) = \bigcap_{\Sigma' \in R\left(\Sigma\right)}H^0\left(X_{\Sigma'}, \O(\lfloor D_{\Sigma'} \rfloor)\right)
\]
is satisfied.
\end{lemma}
\begin{Def}\label{hjhj} Let $\D$ be a toric $b$-divisor which is not necessarily nef and let $\p \colon N_{\Q} \to \Q$ be its associated conical function. We define
\[
\Delta_{\D} \coloneqq \bigcap_{\Sigma' \in R\left(\Sigma\right)}P_{D_{\Sigma'}} = \left\{ m \in M_{\R} \, \big{|}\, \langle m , v_{\tau} \rangle \geq -\p(v_{\tau}), \; \forall v_{\tau} \in N^{\text{prim}} \right\}. 
\]
This is a bounded convex set. 
\end{Def}
We make the following remarks. 
\begin{rem}
(1) If $\D$ is nef, we have that $\Delta_{\D} = \Delta_{\phi_{\D}}$, where $ \Delta_{\phi_{\D}}$ denotes the stability set of the concave extension $\phi_{\D} \colon N_{\R} \to \R$. In this case, we will use either notation interchangeably. 

\noindent(2) In the non-nef case, $\Delta_{\D}$ does not necessarily correspond to a nef toric $b$-divisor as in Proposition \ref{b-convex}. Indeed, any bounded convex body $K$ in $M_{\R}$ is of the form $\Delta_{\D}$ for some toric $b$-divisor $\D$. This follows from the fact that any convex body is the intersection of a (possibly infinite) set of rational half spaces, the so called \enquote{rational supporting hyperplanes}. Hence, while in the nef case we may recover the toric $b$-divisor $\D$ from the convex set $\Delta_{\phi_{\D}}$, in the non-nef case, we may not. This is also true in the classical case: a nef toric divisor $D$ on a smooth and complete toric variety is completely determined by its corresponding polytope $P_D$, whereas in the non-nef case, the polytope does not capture all of the Cartier data of $D$. Nevertheless, the convex set $\Delta_{\D}$ gives us important algebro-geometric information. Like in the classical case, it encodes the global sections of the $b$-divisor, as can be seen by the following proposition.    
\end{rem}
\begin{prop}\label{globalsectionstoric}
Let $\D$ be a toric $b$-divisor on $\X$. Its space of global sections is given by
\[
H^0(\X, \O(\D)) = \bigoplus_{m \in \Delta_{\D} \cap M} k \cdot \chi^m.
\]
\end{prop}
\begin{proof} By Lemma \ref{chupito} and the proof of Proposition \ref{chupo}, the space $H^0(\X, \O(\D))\subseteq k[M]$ is $\T$-stable and can be written as  
\[
H^0(\X, \O(\D)) = \bigoplus_{\chi^m \in H^0(\X, \O(\D))}k \cdot \chi^m. 
\]
We have 
\begin{align*}
\chi^m \in H^0(\X, \O(\D)) \quad &\text{ iff } \quad \div_{\Sigma'}(\chi^m) \geq - \lfloor D_{\Sigma'}\rfloor, \, \forall \Sigma' \in R\left(\Sigma\right)\\ &\text{ iff } \quad \langle m, v_{\tau} \rangle \geq - \lfloor\p(v_{\tau})\rfloor, \, \forall \tau \in \Sigma'(1) \,\text{ and } \forall \Sigma' \in R\left(\Sigma\right) \\ &\text{ iff } \quad\langle m, v_{\tau} \rangle \geq -\lfloor\p(v_{\tau})\rfloor, \, \forall v_{\tau} \in N^{\text{prim}} \\ &\text{ iff }\quad \langle m, v_{\tau} \rangle \geq -\p(v_{\tau}), \, \forall v_{\tau} \in N^{\text{prim}}\\ &\text{ iff } \quad m \in \Delta_{\D}, 
\end{align*}
proving the statement of the proposition. 
\end{proof}

\subsection*{Hilbert--Samuel formula}
The following is a Hilbert--Samuel type formula for toric nef $b$-divisors. 
\begin{theorem}\label{hilbert-samuel}
Let $\D$ be a nef toric $b$-divisor. We denote by $h^0(\ell\D)$ the dimension of the space of global sections $H^0(\X, \ell\D)$ of an integral multiple $\ell\D$ of $\D$. Then we have a Hilbert--Samuel type formula 
\begin{align}\label{hsformula}
\D^n = \lim_{\ell \to \infty} \frac{h^0(\ell\D)}{\ell^n/n!}.
\end{align}
\end{theorem}
\begin{proof}
The so called \emph{lattice volume} $\vol_L$ of a convex set $S \subseteq M_{\R}$ is defined by 
\[
\vol_L (S) \coloneqq \lim_{\ell \to \infty} \frac{ \# \ell S \cap M }{\ell^n}.
\]
In the case where $S$ is a polytope, we have $\vol(S) = \vol_L(S)$ (see e.g. \cite[Section~3]{HKP}). Let $\phi$ be the concave function corresponding to $\D$. Then, since $\D$ is nef, using Proposition \ref{globalsectionstoric} we have that 
\[
\lim_{\Sigma'\in R(\Sigma)}h^0\left(X_{\Sigma'},\ell D_{\Sigma'}\right) = h^0\left(\X, \ell \D\right),
\]
for all $\ell \in \N$. Moreover, the convex set $\Delta_{\D} = \Delta_{\phi}$ can be approximated by polytopes. Since the operator \enquote{$\vol$} is continuous, the sequence of equalities
\begin{align*}
\frac{\D^n}{n!} & = \vol\left(\Delta_{\phi}\right) = \lim_{\Sigma' \in R(\Sigma)}\vol\left(P_{D_{\Sigma'}}\right) \\
& = \lim_{\Sigma' \in R(\Sigma)} \lim_{\ell \to \infty} \frac{\#\ell P_{D_{\Sigma'}}\cap M}{\ell^n} \\ & = \lim_{\ell \to \infty}\lim_{\Sigma' \in R(\Sigma)} \frac{\#\ell P_{D_{\Sigma'}}\cap M}{\ell^n}\\
& = \lim_{\ell\to \infty} \frac{\# \ell \Delta_{\phi}\cap M}{\ell^n}
\end{align*}
is satisfied, thus concluding the proof of the Theorem. 
\end{proof}
\subsection*{Ring of global sections of multiples of a toric $b$-divisor}
As in the classical case, a natural object to study is the ring of global sections of multiples of a toric $b$-divisor. In order to define this, note that given two toric $b$-divisors $\D$ and $\pmb{E}$ together with two rational functions $f$ in $H^0(\X,\D)$ and $g$ in $H^0(\X,\pmb{E})$, it is easy to see that the product $f\cdot g$ is an element in $H^0(\X, \D+\pmb{E})$. Hence, we get a map 
\begin{align*}
H^0(\X,\D) \otimes H^0(\X,\pmb{E}) \longrightarrow H^0(\X, \D+\pmb{E}).
\end{align*}
 \begin{Def}\label{b-gradedring}
Let $\D$ be a nef toric $b$-divisor on $\X$. We define the \emph{ring of global sections of multiples of $\D$} as the graded ring
\[
b\text{-}R(\D) \coloneqq \bigoplus_{\ell \geq 0}H^0\left(\X, \O(\ell\D)\right).
\]
\end{Def}
It is a natural question to ask whether $b\text{-}R(\D)$ is finitely generated as a graded $k$-algebra. We answer this question in the following proposition. 
\begin{prop}\label{fingen}
Let $\D$ be a nef toric $b$-divisor and let $\phi \colon N_{\R} \to \R$ be its corresponding concave function. Then the ring $b\text{-}R(\D)$ is finitely generated if and only if $\D$ is Cartier, i.e.~if $\Delta_{\phi}$ is a rational polytope. 
\end{prop}
\begin{proof}
Let $\D$ be a Cartier toric $b$-divisor. By Proposition \ref{b-convex}, the corresponding convex set $\Delta_{\D} = \Delta_{\phi_{\D}}$ is a rational polytope and it corresponds to the polytope $P_{D_{\Sigma''}}$ of a Cartier toric $\Q$-divisor $D_{\Sigma''}$ on a toric model $X_{\Sigma''}$ for some fan $\Sigma''~\in~R\left(\Sigma\right)$. Clearly, in this case, the ring $b\text{-}R(\D)$ is finitely generated if and only if $R(D_{\Sigma''})$ is finitely generated and the finite generatedness of the latter object is a classical result (see e.g. \cite{E}). \\
To see the \enquote{only if} part, suppose that $b\text{-}R(\D)$ is finitely generated. Let $m_1, \dotsc, m_r$ be elements in the lattice $M$ be such that $\chi^{m_i}$ generate $b\text{-}R(\D)$ for $i=1, \dotsc, r$. 
Let $\tilde{\Sigma} \geq \Sigma$ be sufficiently large so that $\chi^{m_i} \in R\left(D_{\tilde{\Sigma}}\right)$ for all $i=1, \dotsc, r$. Then, for some multiple $\ell \in \N$, we have
\[
P = \ell P_{D_{\tilde{\Sigma}}}, 
\]
where $P =  \convhull(m_1, \dotsc, m_r)$. Hence, the normal fan $\Sigma_P$ of $P$ is a smooth subdivision of $\Sigma$, i.e.~$\Sigma_P \in R\left(\Sigma\right)$. Moreover, for every $\Sigma' \geq \Sigma_P$, we have 
\[
\ell P_{D_{\Sigma'}} = \ell P_{D_{\tilde{\Sigma}}}=P.
\]
Indeed, $P_{D_{\Sigma'}}$ is the convex hull of the generators of global sections. Hence, since we have that $\Delta_{\phi} = \lim_{\Sigma'}P_{D_{\Sigma}} = \frac{1}{\ell}P$, the statement of the proposition follows.
\end{proof}

We end this subsection by formulating the following question regarding the polynomiality of the Hilbert function of a graded algebra. 
\begin{question}
Is it true that the Hilbert function of a \emph{not} finitely generated graded algebra is never a polynomial?
\end{question}

\subsection*{The $b$-Cox ring}
A further interesting object to study is the $b$-Cox ring of a complete, smooth toric variety. Before giving the definition, we recall the notion of the classical Cox ring of a toric variety. 
\begin{Def}
The \emph{Cox ring} of a toric variety $\X$ is the polynomial ring 
\[
\text{Cox}(\X) \coloneqq k\left[\left\{x_{\tau} \, \big{|}\, \tau \in \Sigma(1) \right\}\right].
\]
\end{Def}
This ring is multigraded by the class group $\text{Cl}(\X)$ of $\X$. Indeed, we have the following short exact sequence: 
\[
0 \longrightarrow M \longrightarrow \Z^{\Sigma(1)} \longrightarrow \text{Cl}(\X) \longrightarrow 0,
\] 
where $\Z^{\Sigma(1)}$ is the free group generated by symbols $y_{\tau}$ for $\tau \in \Sigma(1)$. The first morphism sends $m \in M$ to $\text{div}(\chi^m) = \sum_{\tau} \langle m, v_{\tau} \rangle y_{\tau}$ and the second one sends $y_{\tau}$ to the class $[D_{\tau}] \in \text{Cl}(\X)$ of the corresponding toric divisor. The degree of $y_{\tau}$ is just $[D_{\tau}]$. \\

We now define the \emph{$b$-Cox ring} and the \emph{$b$-class group} of $\X$.
\begin{Def}The \emph{$b$-Cox ring} of $\X$ is the following polynomial ring in infinitely many variables: 
\[
\Cox\left(\mathcal{X}_{\Sigma}\right) \coloneqq k\left[\left\{x_v \, \big{|}\, v \in N^{\prim}\right\}\right].
\]
\end{Def}
\begin{Def}
The \emph{$b$-class group} of $\X$ is defined as the inverse limit of class groups of toric varieties 
 \[
\Cl\left(\mathcal{X}_{\Sigma}\right) \coloneqq \varprojlim_{\Sigma' \in R\left(\Sigma\right)}\Cl\left(X_{\Sigma'}\right).
\]
with maps given by the push-forward map of divisor classes of Weil toric divisors. 
The class of a toric $b$-divisor $\D = \left(D_{\Sigma'}\right)_{\Sigma' \in R\left(\Sigma\right)}$ in $\Cl\left(\mathcal{X}_{\Sigma}\right)$ is defined by 
\[
[\D] \coloneqq \big{(}\big{[}\lfloor D_{\Sigma'}\rfloor \big{]}\big{)}_{\Sigma' \in R\left(\Sigma\right)}.
\]
\end{Def}
Note that we are taking only integer coefficients.
As in the classical case, we have a natural grading on $\Cox\left(\mathcal{X}_{\Sigma}\right)$ by the $b$-class group $\Cl\left(\mathcal{X}_{\Sigma}\right)$. Indeed, we have an exact sequence 
\[
0 \longrightarrow M \longrightarrow \Z^{N^{\prim}} \longrightarrow \Cl\left(\mathcal{X}_{\Sigma}\right) \longrightarrow 0
\]
where, as before, the first morphism is given by the assignment 
\[
m \longmapsto (\langle m, v \rangle )_{v \in N^{\prim}},
\]
for $m \in M$, and the second morphism by the assignment
\[
 (m_v)_{v \in N^{\prim}} \longmapsto \left[\sum_vm_vD_v\right] = \left(\left[\sum_{\tau \in \Sigma'(1)}m_{\tau_v}D_{\tau_v}\right]\right)_{\Sigma' \in R\left(\Sigma\right)}.
\]
Given $a \in N^{\prim}$, an element $x^a = \prod_vx_v^{a_v}$ in $\Cox\left(\mathcal{X}_{\Sigma}\right)$, the degree of $x^a$ is defined by
\[
\text{deg}(x^a) \coloneqq \left[\sum_va_vD_v\right] = \left(\left[\sum_{\tau \in \Sigma'(1)}a_{v_{\tau}}D_{v_{\tau}}\right]\right)_{\Sigma' \in R\left(\Sigma\right)} \in \Cl\left(\mathcal{X}_{\Sigma}\right).
\]
Generalizing the result in \cite[Theorem~1.4]{LV} for the case of $b$-divisors, we have the following proposition. 
\begin{prop}
The vector space $H^0(X_{\Sigma}, \D)$ of global sections of $\D$ is isomorphic to the degree $[\D]$ part of $\Cox\left(\mathcal{X}_{\Sigma}\right)$.
\end{prop}
\begin{proof}
Let $\D$ be a toric $b$-divisor given by  $\D = \sum_{i \in \N}a_{v_i}D_{v_i}$. A rational function $\chi^m$ ($m\in M$) is in $H^0(\X, \D)$ if and only if $\text{div}(\chi^m) + \lfloor\D\rfloor \geq 0$, or equivalently, if $\langle m, v_i \rangle \geq -\lfloor a_{v_i}\rfloor$ for all $i \in \N$. Hence, the expression 
\[
x_1^{\langle m,v_1 \rangle + \lfloor a_{v_1}\rfloor} \dotsm x_d^{\langle m,v_d \rangle + \lfloor a_{v_d}\rfloor} \dotsm
\]
has only nonnegative exponents and this defines a monomial of degree $[\D]$ in $\Cox\left(\mathcal{X}_{\Sigma}\right)$. Conversely, let $x^s = \left(x_{v_i}^{s_{v_i}}\right)_{i \in \N}$ be a monomial of degree $[\D]$ in $\Cox\left(\mathcal{X}_{\Sigma}\right)$. There exists an $m \in M$ such that 
\[
\sum_is_{v_i}D_{v_i} = \sum_i a_{v_i}D_{v_i} + \div \left(\chi^m\right).
\] Hence,
\[
\div \left(\chi^m\right) + \lfloor\D\rfloor = \langle m, v_i\rangle + \lfloor a_{v_i} \rfloor = s_{v_i} \geq 0,
\] which implies that $m \in H^0(\X, \D)$. This correspondence is obviously bijective. 
\end{proof}\begin{rem} It is a classical result (see e.g \cite[Theorem~12.5.3]{CLS}) that the Chow ring $\text{A}_{\Q}^{\bullet}\left(\X\right)$ of a toric variety $\X$ is a quotient of the Cox ring. Specifically, we have the following:
\[
\text{A}_{\Q}^{\bullet}\left(\X\right) = \text{Cox}\left(\X\right) / (\mathcal{I} + \mathcal{J}),
\]
where $\mathcal{I}$ is the monomial ideal with square-free generators 
\[
\mathcal{I} \coloneqq \left\langle x_{i_1} \dotsm x_{i_s} \, \big{|}\, i_j \text{ are distinct and }\text{cone}(\tau_{i_1},\dotsc,\tau_{i_s}) \text{ is not a cone of }\Sigma \right\rangle
\]
and  $\mathcal{J}$ is the ideal generated by the linear forms 
\[
\mathcal{J} \coloneqq \left\langle \sum_{\tau \in \Sigma(1)} \langle m, v_{\tau}\rangle x_{\tau} \, \bigg{|}\, m \in M\right\rangle.
\]
The $b$-Cox ring of a complete toric variety is thus the natural object where one hopes to develop a suitable intersection theory of toric $b$-divisors other than just for top degrees. 
\end{rem}

\section{Relationship with Okounkov bodies}\label{section6}
Okounkov bodies (in the literature often called Newton--Okounkov bodies) are convex bodies which one can attach to an algebraic variety together with some extra data, e.g.~a divisor. These convex bodies have been widely and successfully used to explore the geometry of the variety using convex geometrical methods (see \cite{OK1,OK2} and also \cite{KK, KK2, KKa} and \cite{LM} and the references therein). The main goal of this section is to identify the convex sets $\Delta_{\D}$ arising from toric $b$-divisors $\D$ (Definition~ \ref{hjhj}) with convex Okounkov bodies attached to algebras of almost integral type defined by K.~Kaveh and A.~G.~Khovanskii in \cite{KK}. As an application, in the big and nef case, we also construct a global Okounkov body generalizing the global Okounkov body associated to a big divisor constructed in  \cite{LM} to the $b$-setting. 

\subsection*{Identification of $b$-convex bodies and Okounkov bodies}\label{okounkov_algebra}
Throughout this section, $\Sigma \subseteq N_{\R}$ will denote a complete and smooth fan. A \emph{convex body} is a compact, convex set in $\R^n$ with non-empty interior. We start by recalling the definition of the convex Okounkov body associated to an algebra of almost integral type defined in \cite{KK} (where it is called the Newton--Okounkov body). Later on, we will apply this construction to the case of the ring of global sections of multiples of a (not necessarily nef) toric $b$-divisor. This turns out to be an algebra of almost integral type. Then we give an isomorphism of lattices $M \simeq \Z^n$ which identifies the convex set $\Delta_{\D}$ of Definition \ref{hjhj} with the associated Okounkov body.

\subsubsection*{Okounkov body associated to an algebra of almost integral type}
We start with some definitions which are taken from \cite[Section~2.3]{KK}. 
\begin{Def}
Let $F$ be a finitely generated field of transcendence degree $n$ over $k$. For any integer $\ell \geq 0$, a homogeneous element of degree $\ell$ in $F[t]$ is an element $a_{\ell}t^{\ell}$ where $a_{\ell} \in F$. Let $B$ be a $k$-linear subspace of $F[t]$. The collection $B_{\ell}$ of homogeneous elements of degree $\ell$ in $B$ is a $k$-linear subspace of $B$ called the $\ell$-th homogeneous component of $B$. The $k$-linear space $L_{\ell} \subseteq F$ such that $a \in L_{\ell}$ if and only if $at^{\ell} \in B_{\ell}$ is called the $\ell$-th homogeneous subspace of $B$. A $k$-linear subspace $B \subseteq F[t]$ is called \emph{graded} if it is the direct sum of its homogeneous components. A subalgebra $A \subseteq F[t]$ is called \emph{graded} if it is graded as a linear subspace of $F[t]$. 
\end{Def}  
\begin{Def}\label{def:algebras}
We define the following three classes of graded subalgebras of $F[t]$.
\begin{enumerate}
\item To each non-zero finite dimensional linear subspace $L \subseteq F$ over $k$, we associate the graded subalgebra $A_L \subseteq F[t]$ defined as follows: it's zeroth homogeneous component is $k$ and for each $\ell >0$, its $\ell$-th subspace is $L^{\ell}$, the subspace spanned by all products $f_1\cdots f_{\ell}$ with $f_1, \dotsc, f_{\ell} \in L$. That is, $A_L = \bigoplus_{\ell \geq 0}L^{\ell}t^{\ell}$.
\item A graded subalgebra $A \subseteq F[t]$ is called an \emph{algebra of integral type} if there is a graded subalgebra $A_L \subseteq F[t|$ for some non-zero finite dimensional linear subspace $L \subseteq F$ over $k$ such that the subalgebra $A$ is a finitely generated $A_L$-module (equivalently, $A$ is finitely generated over $k$ and is a finite module over the subalgebra generated by $A_1$, the first homogeneous components of $A$).
\item Finally, a graded subalgebra $A \subseteq F[t]$ is called an \emph{algebra of almost integral type} if there is an algebra $A'\subseteq F[t]$ of integral type such that $A \subseteq A'$ (equivalently, $A\subseteq A_L$ for some non-zero finite dimensional linear subspace $L \subseteq F$). 
\end{enumerate}
\end{Def}
\begin{rem}
We can think of the $A_L$'s as being the homogeneous coordinate rings of projective varieties. Then, one can see that rings of global sections of ample line bundles give rise to algebras $A$ of integral type, whereas rings of global sections of arbitrary divisors give rise to algebras of almost integral type (\cite[Theorem~3.7 and 3.8]{KK}). It turns out that also toric $b$-divisors give rise to algebras of the latter type.
\end{rem}
\begin{Def} Let $A$ be an algebra over $k$ and let $G$ be an ordered abelian group, i.e.~an abelian group equipped with a total order \enquote{$<$} with respect to the group operation. A \emph{valuation} on $A$ is a function 
\[
\nu \colon A\setminus \{0\} \longrightarrow G
\]
satisfying the following conditions: \begin{enumerate}
\item For all $f,g \in A$ with $f,g,f+g \neq 0$, we have $\nu(f+g) \geq \min\left(\nu(f), \nu(g)\right)$. 
\item For all $0 \neq f \in A$ and $0 \neq \lambda \in k$, we have $\nu(\lambda f) = \nu(f)$. 
\item For any $f,g \in A$ with $f,g \neq 0$, we have $\nu(fg) = \nu(f) + \nu(g)$. 
\end{enumerate}
A valuation $\nu$ is said to be \emph{faithful} if its image is the whole of $G$. It is said to \emph{have $1$-dimensional leaves} if it satisfies that if whenever $\nu(f) = \nu(g)$ for $f,g \in F\setminus \{0\}$, then 
\[
\nu(g + \lambda f) > \nu(g)
\]
for some $\lambda \in k$. 
\end{Def}
We associate a semigroup $S(A) \subseteq \Z^n \times \Z$ to an algebra $A \subseteq F[t]$ of almost integral type. This is done as follows: first, assume we are given a faithful valuation 
\[
\nu \colon F\setminus \{0\} \longrightarrow \Z^n
\]
with $1$-dimensional leaves (we assume $\Z^n$ to be equipped with a total order \enquote{$<$} which respects addition). 
We now consider the total ordering $\prec$ on the group $\Z^n \times \Z$ given in the following way: let $(\alpha, i)$, $(\beta,j)$ $\in \Z^n \times \Z$.
\begin{enumerate}
\item If $i<j$, then $(\alpha,i) \prec (\beta,j)$. 
\item If $i=j$ and $\alpha < \beta$, then $(\alpha,i) \prec (\beta,j)$. 
\end{enumerate}
\begin{Def} Given a valuation $\nu$ on $F$ as above, we define a valuation
\[
\nu_t \colon F[t]\setminus \{0\} \longrightarrow \Z^n \times \Z
\]
in the following way: let $p(t) = a_{i}t^{i} + \cdots + a_0$, with $a_{i} \neq 0$, be a polynomial in $F[t]$. Then we put 
\[
\nu_t(p) \coloneqq \left(\nu(a_{i}), i\right)
\]
and extend $\nu_t$ to $F(t)\setminus\{0\}$. This is a faithful valuation with $1$-dimensional leaves  extending $\nu$ on $F$ which we also denote by $\nu_t$. 
\end{Def}
We are now ready to define the semigroup $S(A) \subseteq \Z^n \times \Z$ associated to an algebra $A \subseteq F[t]$ of almost integral type.
\begin{Def}\label{semigroup}
Let $A \subseteq F[t]$ be an algebra of almost integral type and assume that we are given a faithful valuation $\nu$ on $F$ with $1$-dimensional leaves. We define the semigroup $S(A)_{\nu}$ by 
\[
S(A)_{\nu} \coloneqq \nu_t(A\setminus \{0\}) \subseteq \Z^n \times \Z.
\] 
\end{Def}
We sometimes write $S(A) = S(A)_{\nu}$ keeping in mind the dependence on the valuation $\nu$. 
This semigroup has some nice properties. 
\begin{Def}
A pair $(S,R)$ where $S$ is a semigroup in $\Z^n$ and $R$ is a rational half-space in 
\[
L(S) \coloneqq \text{real span of } S \subseteq \R^n
\] 
is said to be \emph{admissible} if $S \subseteq R$. An admissible pair is called \emph{strongly admissible} if, moreover, the cone over $S$ defined by 
\begin{align}\label{cone_over_S}
\cone(S) \coloneqq \overline{\convhull\left( S \cup \{0\} \right) } \subseteq \R^n 
\end{align} 
is strictly convex and intersects the boundary $\partial R$ of $R$ only at the origin.
\end{Def}
\begin{Def}
A \emph{non-negative} semigroup of integral points in $\R^{n+1}$ is a semigroup 
\[
S \subseteq \R^n \times \R_{\geq 0}
\]
which is not contained in the hyperplane $x_{n+1} = 0$. To such a semigroup one can associate an admissible pair $(S, R(S))$ by setting 
\[
R(S) \coloneqq L(S) \cap (\R^n\times \R_{\geq 0}).
\] 
A non-negative semigroup $S$ is said to be strongly admissible if its associated admissible pair $(S, R(S))$ has this property.
\end{Def}
The following is \cite[Theorem~2.31]{KK}.
\begin{lemma}
The semigroup $S(A) \subseteq \Z^{n+1}$ of Definition \ref{semigroup} is non-negative and its associated admissible pair $(S(A), R(S(A)))$ is strongly admissible.
\end{lemma}
We are now ready to define the convex Okounkov body associated to an algebra of almost integral type. 
\begin{Def}
Let $A \subseteq F[t]$ be an algebra of almost integral type and let $S(A)$ be the non-negative, strongly admissible semigroup defined above. Consider the strongly convex cone $C$ given by
\[
C = \cone(S(A)) \subseteq \R^n \times \R.
\] 
The \emph{Okounkov body} $\Delta_A$ associated to $A$ is then defined to be the slice of $C$ at height $1$, i.e.
\[
\Delta_A \coloneqq C \cap \left(\R^n \times \{1\}\right).
\]
\end{Def}
\subsubsection*{Okounkov body associated to the ring of global sections of multiples of a toric $b$-divisor}
We now proceed to show that the ring of global sections of multiples of a toric $b$-divisor $\D$ defines an algebra of almost integral type $A_{\D}$ and that its associated Okounkov body $\Delta_{A_{\D}}$ can be identified with the convex set $\Delta_{\D}$ associated to $\D$ from Definition \ref{hjhj}.

Let $F =k(\X)$ be the field of rational functions of the complete, smooth toric variety $\X$. This is the quotient field of $k[M]$, the ring of Laurent polynomials. Hence, elements of $F$ are quotients of elements of the form $\sum_{m\in M}a_m\chi^m$, where, as usual, $\chi^m$ denotes the character of the torus of weight $m \in M$. Recall from the previous section that to a toric $b$-divisor $\D$ one can associate the space of global sections $H^0(\X, \D) \subseteq k[M]$ which is given by 
\[
H^0(\X, \D) = \{f \in F \,|\, b\text{-div}(f) + \lfloor\D \rfloor \geq 0\}.
\]
Moreover, we have a well defined map 
\begin{align}\label{graded}
H^0(\X,\D) \otimes H^0(\X,\pmb{E}) \longrightarrow H^0(\X, \D+\pmb{E})
\end{align}
for any toric $b$-divisors $\D$ and $\pmb{E}$.
\begin{Def}\label{algebra}
We define the set $A_{\D}$ to be the collection of all polynomials $f(t) = \sum_{\ell}f_{\ell}t^{\ell}$ with $f_{\ell} \in H^0(\X, \ell \D)$, i.e.
\[
A_{\D} = \bigoplus_{\ell \geq 0}H^0(\X, \ell \D)t^{\ell}.
\] 
By (\ref{graded}), $A_{\D} \subseteq F[t]$ is a graded subalgebra.
\end{Def}
\begin{rem} Note that the graded subalgebra $A_{\D}$ is like the graded algebra $b\text{-}R(\D)$ from Definition \ref{b-gradedring}, with the difference that in $A_{\D}$ we are taking track of the grading with the variable~$t$. We choose to use this notation in order to be compatible with the notation in \cite{KK}.
\end{rem}
\begin{prop}\label{prop:almostintegral}
Let $\D$ be a toric $b$-divisor on $\X$. Then the graded subalgebra $A_{\D} \subseteq F[t]$ is an algebra of almost integral type.
\end{prop}
\begin{proof}
For any fan $\Sigma' \in R(\Sigma)$ we have the inclusion of algebras
\[
A_{\D} \subseteq A_{D_{\Sigma'}}.
\]
 Indeed, $f \in H^0(\X, \D)$ if and only if $f \in H^0(\X, D_{\Sigma'})$ for all fans $\Sigma' \in R(\Sigma)$. The claim now follows from the fact that $A_{D_{\Sigma'}} \subseteq F[t]$ is finitely generated and hence an algebra of integral type (see \cite{E}). 
\end{proof}
We want to describe the Okounov body $\Delta_{A_{\D}}$ associated to the algebra of almost integral type $A_{\D}$. Recall from the previous section that the convex body $\Delta_{A_{\D}}$ depends on the choice of a faithful valuation $\nu \colon F\setminus \{0\} \to \Z^n$ with $1$-dimensional leaves. Our goal is to give an identification of lattices $\phi \colon M \to \Z^n$ and a faithful valuation $\nu \colon F\setminus \{0\} \to \Z^n$ with $1$-dimensional leaves such that, if we denote by $\phi_{\R} \colon M_{\R} \to \R^n$ the induced isomorphism of real vector spaces, then we get an identification
\[
\Delta_{A_{\D}} = \phi_{\R}(\Delta_{\D})
\]
of convex bodies.  

In order to do this, we follow a construction presented in \cite[Section~6.1]{LM}. First, let us start with a complete flag 
\[
Y_{\bullet} = Y_0 \supset Y_1 \supset \dotsb \supset Y_n = \{\text{pt}\} 
\]
of $\T$-invariant subvarieties of $\X$. We have $\text{codim}_{\X}(Y_i) = i$ for $i = 0, \dotsc, n$. Since we are assuming $\X$ to be smooth, we can order the prime $\T$-invariant divisors $D_1, \dotsc , D_r$ of $\X$, where $r = \# \Sigma(1)$, in such a way that we have $Y_i = D_1 \cap \dotsb \cap D_i$ for $i \leq n \leq r$. 

For $i = 1, \dotsc, r$, we denote by $v_i$ the primitive generator of the ray corresponding to $D_i$. Then the set of vectors $\{v_1, \dotsc, v_n\}$ forms a basis of $N$ and generates a cone $\sigma$ in $\Sigma$ of maximal dimension $n$. We get an isomorphism 
\begin{align}\label{ident}
\phi \colon M \simeq \Z^n, 
\end{align}
given by 
\[
u \longmapsto (\langle u, v_i \rangle )_{1 \leq i \leq n}.
\] 
We denote by $\phi_{\R} \colon M_{\R} \simeq \R^n$ the induced isomorphism of real vector spaces. 
\begin{Def}\label{valuation}
Let $F$ and $Y_{\bullet}$ be as above. We define the valuation $\nu = \nu_{Y_{\bullet}} \colon F \setminus \{0\} \to \Z^n$ in the following way: let $s \in k[M]$ and let $\div(s) = \sum_{i=1}^ra_iD_i$ be the zero locus of $s$. Then $\nu(s)$ is defined to be the tuple $(a_1, \dotsc, a_n)$.  Finally, we extend by linearity to $F \setminus \{0\}$.
\end{Def}
One can check that 
\[
\# \Im\left(\nu \colon H^0(\X, \ell \D) \setminus \{0\} \longrightarrow \Z^n\right) = h^0(\X, \ell \D)
\]
for any integer multiple $\ell \geq 0$. Moreover, by \cite{KK}, the valuation $\nu$ is faithful and has $1$-dimensional leaves. We can now state the main theorem of this section. 
\begin{theorem}\label{equal_conv}
Let notations be as above and consider a toric $b$-divisor $\D = (D_{\Sigma'})_{\Sigma' \in R(\Sigma)}$ such that ${D_{\Sigma}}|_{U_{\sigma}}$ is trivial. Let $A_{\D} \subseteq F[t]$ be the algebra of almost integral type associated to $\D$ from Definition \ref{algebra}. Let $\nu \colon F \setminus \{0\} \to \Z^n$ be the valuation of Definition \ref{valuation} and let $\phi_{\R}$ be the isomorphism (\ref{ident}). Then we can identify the convex sets $\Delta_{A_{\D}} = \Delta_{A_{\D}}^{\nu}$ and $\Delta_{\D}$, i.e.~we have that
\[
\Delta_{A_{\D}} = \phi_{\R}(\Delta_{\D}).
\]
\end{theorem}
\begin{proof}
Recall from the previous section that the semigroup $S(A_{\D})$ is defined to be the image $\nu_t(A_{\D} \setminus \{0\}) \subseteq \Z^n \times \Z$. Also, recall that the strictly convex cone $C\subseteq \R^n \times \R$ is given by
\[
C = \cone(S(A_{\D})) = \overline{\convhull\left(S(A_{\D}) \cup \{0\}\right)}.
\]
Consider the convex body $\phi_{\R}\left(\Delta_{\D}\right) \subseteq \R^n \times \R$ lying in the slice $\R^n \times \{1\}$. We define the cone $C' \subseteq \R^n \times \R$ by  
\[
C' \coloneqq \text{strictly convex cone in }\R^n \times \R \text{ over the convex set } \phi_{\R}(\Delta_{\D}).
\]
Note that we have $ \phi_{\R}(\Delta_{\ell \D}) = \phi_{\R}(\ell \Delta_{\D}) = \ell \phi_{\R}(\Delta_{\D})$. Hence, $C'$ is characterized by 
\[
C'\cap \left(\R^n \times \{\ell\}\right) = \phi_{\R}(\Delta_{\ell \D}).
\]
In order to prove the theorem it suffices to show that $C = C'$. Moreover, since $C \text{ (resp.~}C')$ is the convex hull of its lattice points, it suffices to show that $C$ and $C'$ contain the same lattice points. Indeed, let $(\phi(m), \ell) \in C' \cap \left(\Z^n \times \Z\right)$ so that $m \in \Delta_{\ell\D}$. The zero locus of the corresponding section $\chi^m \in H^0(\X, \ell \D)$ is $\ell \D + \sum_{i=1}^r \langle m, v_i\rangle D_i$. By assumption, we have that the coefficient of $D_i$ in $\D$ is $0$ for $i = 1, \dotsc, n$, since $D_{\Sigma}|_{U_{\sigma}}$ is trivial. Hence, $\nu(\chi^m) = \phi(m)$ and thus $(\phi(m), \ell) \in C \cap \left(\Z^n \times \Z\right)$. Now, since $\phi$ is injective and we have precisely $h^0(\X, \ell \D)$ lattice points in $\Delta_{\ell\D}$, we have that for all $\ell \geq 0$,
\[
C \cap \left(\Z^n \times \{\ell\}\right) = \Im\left(v \colon H^0(\X, \ell \D) \setminus \{0\} \longrightarrow \Z^n\right) = \phi(\Delta_{\ell\D} \cap M) = C' \cap \left(\Z^n \times \{\ell\}\right), 
\]
and hence 
\[
C \cap \left(\Z^n \times \Z \right) = C' \cap \left(\Z^n \times \Z \right),
\]
as we wanted to show. 
 
\end{proof} 
\begin{rem}
The assumption in the above theorem that ${D_{\Sigma}}|_{U_{\sigma}}$ is trivial can always be achieved by passing to a linearly equivalent divisor $D_{\Sigma} + \div\left(\chi^{m}\right)$. The corresponding polytope $P_{D_{\Sigma}}$ is then translated accordingly to $P_{D_{\Sigma}} -m$.
\end{rem}
\subsection*{Applications}
Throughout this section, $\Sigma \subseteq N_{\R}$ will denote a complete and smooth fan of dimension $n$, i.e.~such that $\dim (\X) = n$. Furthermore, $\D$ will denote a toric $b$-divisor on $\X$ and $A_{\D} \subseteq F[t]$ its associated algebra of almost integral type. We now give some applications of the identification of the convex sets $\Delta_{\D}$ with Okounkov bodies. First, we make a statement regarding the growth of the Hilbert function of the graded subalgebra $A_{\D}$. Then, in the big and nef case, we construct a \emph{global Okounkov $b$-body} generalizing the global Okounkov body constructed in \cite{LM}. This is an interesting tool which one can use in order to deeper study the birational geometry of toric varieties.
   
\subsubsection*{Asymptotic growth of Hilbert functions of algebras of almost integral type}
Let $H_{A_{\D}}$ be the Hilbert function of the algebra $A_{\D}$. Before describing the asymptotic behavior of $H_{A_{\D}}$, we make the following remark. 
\begin{rem} The Hilbert function $H_{S} \colon \N \to \N,$ of a semigroup $S \subseteq \Z^n \times \Z$ is given by the assignment 
\[
 t \longmapsto \#\cone(S) \cap \left(\Z^n \times \{t\}\right).
\]
By  \cite[Proposition~2.27]{KK} the function $H_{A_{\D}}$ coincides with $H_{S(A_{\D})}$. Moreover, recall the Ehrhart function $f_{\D}=f_{\Delta_{\D}}\colon \N \to \N$ of the convex set $\Delta_{\D}$ (Definition \ref{ehrhartfunction}). Then the proof of Theorem \ref{equal_conv} implies that 
\[
f_{\D} = H_{S(A_{\D})} = H_{A_{\D}}.
\]
\end{rem}
We have the following proposition. 
\begin{prop}\label{fgf}
Let $\D$ be a toric $b$-divisor satisfying 
\[
\vol\left(\Delta_{\D}\right) > 0.
\]
Then the Hilbert function $H_{A_{\D}}(t)$ of the graded subalgebra $A_{\D}$ grows like $a_nt^n$, where the $n$-th growth coefficient $a_n$ is equal to the volume $\vol\left(\Delta_{\D}\right)$. 
\end{prop}
\begin{proof} This follows from \cite[Corollary~3.11]{KK} and Theorem \ref{equal_conv}.
\end{proof}

\subsubsection*{The global convex $b$-cone}
We construct a \emph{global} Okounkov $b$-body in the big and nef case, generalizing the global Okounkov body associated to a projective variety together with a big divisor constructed in \cite{LM}.  

For any smooth and complete variety $X$ of dimension $n$, we denote by $N^1(X) \subseteq H^2(X, \R)$ the N\'eron--Severi group of numerical equivalence classes of divisors on $X$. This is a free abelian group of finite rank. The corresponding finite dimensional $\Q$- and $\R$-vector spaces are denoted by $N^1(X)_{\Q}$ and $N^1(X)_{\R}$, respectively. 

For a pair $(X,D)$ consisting of a smooth $n$-dimensional projective variety together with a big divisor $D\subseteq X$, the authors in \cite{LM} construct a convex Okounkov body $\Delta^{Y_{\bullet}}_D \subseteq \R^n$. The construction depends on the choice of a complete flag $Y_{\bullet}: X=Y_0 \supset Y_1 \supset \cdots \supset Y_n = \{pt\}$ in $X$. Here it is also shown that $\Delta^{Y_{\bullet}}_D$ is independent of the choice of the numerical class $[D]$ in $N^1(X)$ (see \cite[Proposition~4.1]{LM}). The construction of the Okounkov body $\Delta^{Y_{\bullet}}_D$ can be seen as a special case of the construction described in the previous section. Indeed, given a complete flag $Y_{\bullet}$, one can define a valuation $\nu_{Y_{\bullet}}$ induced by this flag which is similar to the valuation given in Definition \ref{valuation} (see \cite[Section~1.1]{LM}). Then the convex bodies $\Delta^{Y_{\bullet}}_D$ and $\Delta_{A_D}$ coincide (see \cite[Theorem~3.9]{KK}).  Moreover, in \cite{LM} the authors show that these convex bodies fit together in a nice way while $[D]$ runs through the set of numerical equivalence classes. The following is \cite[Theorem~4.5]{LM}. 
\begin{theorem}\label{global_body}
Let $X$ be a smooth, projective variety of dimension $n$. There exists a closed convex cone 
\[
\Delta(X) \subseteq \R^n \times N^1(X)_{\R}
\]
characterized by the property that in the diagram 
\begin{center}
    \begin{tikzpicture}
      \matrix[dmatrix] (m)
      {
        \Delta(X) &  & \R^n \times N^1(X)_{\R}\\
        & N^1(X)_{\R} & \\
      };
      \draw[->] (m-1-1) to (m-2-2);
      \draw[right hook->] (m-1-1) to (m-1-3);
      \draw[->] (m-1-3) to node[below right]{$\pr_2$} (m-2-2);
     \end{tikzpicture}
   \end{center}
   the fiber of $\Delta(X)\to N^1(X)_{\R}$ over any big class $[D] \in N^1(X)_{\Q}$ is $\Delta_{A_D}$, i.e.~we have that 
   \[
   \pr_2^{-1}([D])\cap\Delta(X) = \Delta_{A_D} \subseteq \R^n \times \{[D]\} = \R^n.
   \]
   The convex cone $\Delta(X)$ is referred to as the \emph{global Okounkov body} of $X$. 
\end{theorem}
We have the following remark describing the image of $\Delta(X)$ in $N^1(X)_{\R}$.
\begin{rem}\label{big_cone} The image of $\Delta(X)$ in $N^1(X)_{\R}$ is the pseudo-effective cone $\overline{\text{Eff}}(X)$, i.e.~the closure of the cone spanned by all numerical classes of effective divisors, whose interior is the big cone $\BBig(X)$ (see the proof of \cite[Theorem~4.5]{LM}).
\end{rem}
We now describe the global Okounkov body $\Delta(\X)$ of the smooth and complete toric variety $\X$. For this, we follow mainly \cite[Section~6.1]{LM}. One can show that starting with a choice of a complete flag $Y_{\bullet}$ consisting of torus invariant subvarieties and the choice of a basis of $M \simeq \Z^n$ inducing the isomorphism $\phi_{\R} \colon M_{\R} \simeq \R^n$ from (\ref{ident}) and assuming that $D|_{U_{\sigma}}$ is trivial, one has that 
\[
\Delta^{Y_{\bullet}}_D = \phi_{\R}(P_D)
\]
for any big toric divisor $D \subseteq \X$. Here, as usual, $P_D$ denotes the polytope associated to the toric divisor $D$ (see Section \ref{notations}). Furthermore, we have the short exact sequence 
\[
0 \longrightarrow M \longrightarrow \Z^{r}\overset{q}{\longrightarrow} \text{Cl}(\X) \longrightarrow 0,
\]
where $r = \#\Sigma(1)$. Also, in this special toric case, linear equivalence agrees with numerical equivalence, i.e.~we have that $N^1(\X) = \text{Cl}(\X) = \text{Pic}(\X)$ has no torsion. The choice of a basis for $M$ gives us the dual basis for $N = M^{\vee}$. This in turn gives us a splitting of the above exact sequence and consequently, an isomorphism 
\[
\beta \colon \Z^n \times N^1(\X) \simeq \Z^{r},
\]
such that $\beta^{-1}(D) = \left(\pr_1(D), q(D)\right)$, where $\pr_1 \colon \Z^r \to \Z^n$ denotes the projection onto the first $n$ components. By \cite[part 2 of Proposition~6.1]{LM} we have the following description of the global Okounkov body $\Delta(\X)$ of Theorem \ref{global_body} in the toric case: the convex set $\Delta(\X)$ is the inverse image under the isomorphism 
\[
\beta_{\R} \colon \R^n \times N^1(\X)_{\R} \simeq \R^{r}
\]
of the non-negative orthant $\R^{r}_+ \subseteq \R^{r}$.

Our goal is to extend Theorem \ref{global_body} to the case of toric big and nef $b$-divisors. Let us start by giving the definition of a \emph{big} toric $b$-divisor.
\begin{Def}\label{bigbdiv}
A toric $b$-divisor $\D = (D_{\Sigma'})_{\Sigma' \in R(\Sigma)}$ is \emph{big} if it has positive volume, i.e.~if 
\[
\vol(\D) \coloneqq \limsup_{\ell \to \infty} \frac{n! h^0(\X, \ell\D)}{\ell^n} >0.
\]
\end{Def}
\begin{rem} It follows from the definitions that a toric $b$-divisor $\D$ is big if and only if its corresponding convex set $\Delta_{\D}$ is full-dimensional. This generalizes the classical theory of toric divisors, where it is known that bigness of a toric divisor $D$ is equivalent to the full-dimensionality of the polyhedron $P_D$ (see the discussion after \cite[Lemma~9.3.0]{CLS}).
\end{rem}
\begin{rem} Let $\D$ be a nef toric $b$-divisor on $\X$. We can ask whether bigness of $\D$ is equivalent to each of the incarnations $D_{\Sigma'}$, for $\Sigma' \in R(\Sigma)$, being big. One implication is clear, since if we start with a big toric $b$-divisor, then if one of the polytopes $P_{D_{\Sigma'}}$ where not to have full-dimensional volume, then since $\Delta_{\D} \subseteq P_{D_{\Sigma'}}$, this would contradict the bigness hypothesis. The other implication however is not true, i.e.~we can produce a nef toric $b$-divisor, all whose incarnations are big, but it itself is not. This follows from the fact that we can produce a sequence of concave virtual support functions compatible under push-forward, with full-dimensional stability sets, converging pointwise to a concave function whose stability set is of strictly smaller dimension. Indeed consider the following sequence of convex sets and the corresponding support functions. This converges to the $1$-dimensional segment $\left[(0,1),(1,0)\right]$ in $\R^2$. 
\begin{center}
\begin{tikzpicture}[scale=0.8]{h}
\filldraw[yellow] (0,0)--(2,0)--(0,2)--cycle;
\filldraw[yellow, opacity = 0.2](4,0)--(4,2)--(6,0)--cycle;
\filldraw[yellow, opacity = 0.2](8,0)--(8,2)--(10,0)--cycle;
\filldraw[yellow, opacity = 0.2] (12,0)--(12,2)--(14,0) --cycle; 
\filldraw[yellow] (5,0) -- (4,2) -- (6,0) -- cycle;
\filldraw[yellow] (26/3,2/3)--(8,2)--(10,0)--cycle;
\filldraw[yellow] (66/5,2/5)--(12,2)--(14,0) --cycle ;
\draw[dashed, opacity = 0.3] (4,2) -- (5,0);
\draw[dashed, opacity = 0.3] (8,2) -- (9,0);
\draw[dashed, opacity = 0.3] (10,0) -- (8,1);
\draw[dashed, opacity = 0.3] (14,0) -- (12,1);
\draw[dashed, opacity = 0.3] (12,2) -- (13, 0);
\draw[dashed, opacity = 0.3] (12,2) -- (13.5, 0);
\draw[->] (2,1)-- (3.5,1);
\draw[->] (6,1)-- (7.5,1);
\draw[->] (10,1) -- (11.5,1);
\draw[->] (14,1) -- (15.5,1); 
\draw (16,1) node{$\cdots$};
\end{tikzpicture}
\end{center}
\end{rem} 

\begin{Def}\label{num_class}
We define the space of numerical classes $N^1(\mathfrak{X}_{\Sigma})$ of the toric Riemann--Zariski space $\mathfrak{X}_{\Sigma}$ as the inverse limit 
\[
N^1(\mathfrak{X}_{\Sigma}) \coloneqq \varprojlim_{\Sigma'\in R(\Sigma)} N^1(X_{\Sigma'})
\]
with maps given by the push-forward map of numerical classes of Cartier divisors.  Moreover, we define the topological space $N^1(\mathfrak{X}_{\Sigma})_{\R}$ by 
\[
N^1(\mathfrak{X}_{\Sigma})_{\R} \coloneqq \varprojlim_{\Sigma'\in R(\Sigma)} N^1(X_{\Sigma'})_{\R}.
\]
\end{Def}
\begin{rem}
Recall that in the case of smooth toric varieties numerical equivalence of divisors agrees with linear equivalence. Hence, the following definition makes sense. 
\end{rem}
\begin{Def}\label{defpicardtoric}
The Picard group $\Pic(\mathfrak{X}_{\Sigma})$ of the toric Riemann--Zariski space $\mathfrak{X}_{\Sigma}$ is defined as the inverse limit 
\[
\Pic(\mathfrak{X}_{\Sigma}) \coloneqq \varprojlim_{\Sigma'\in R(\Sigma)} \Pic(X_{\Sigma'})
\]
with maps given by the push-forward map of linear (= numerical) equivalence classes of Cartier divisors. Moreover, we define the topological space $\Pic(\mathfrak{X}_{\Sigma})_{\R}$ by 
\[
\Pic(\mathfrak{X}_{\Sigma})_{\R}  \coloneqq \varprojlim_{\Sigma'\in R(\Sigma)} \Pic(X_{\Sigma'})_{\R}.
\]
\end{Def}
\begin{rem}
(1) Since in the case of smooth toric varieties numerical equivalence of divisors agrees with linear equivalence, we have that this is also true for the toric Riemann--Zariski space, i.e.~we have that
\[
N^1(\mathfrak{X}_{\Sigma}) = \Pic(\mathfrak{X}_{\Sigma}).
\]
We will write $[\D]$ to denote the class of the toric $b$-divisor $\D$ in $N^1(\mathfrak{X}_{\Sigma})$.

\noindent (2) Note that $N^1(\mathfrak{X}_{\Sigma})_{\R}$ is an inverse limit of finite dimensional vector spaces, but it is itself not finite dimensional. We view it just as a topological vector space (with the topology induced by the inverse limit topology). 
\end{rem}
The following is the analogue of Theorem \ref{global_body} for toric big and nef $b$-divisors. 
\begin{theorem}\label{the:globalconvexset}
There exists a closed convex cone 
\[
\Delta(\mathfrak{X}_{\Sigma}) \subseteq \R^n \times N^1(\mathfrak{X}_{\Sigma})_{\R} 
\]
characterized by the property that in the diagram 
\begin{center}
    \begin{tikzpicture}
      \matrix[dmatrix] (m)
      {
        \Delta(\mathfrak{X}_{\Sigma}) & & \R^n \times N^1(\mathfrak{X}_{\Sigma})_{\R}\\
        & N^1(\mathfrak{X}_{\Sigma})_{\R} & \\
      };
      \draw[->] (m-1-1) to (m-2-2);
      \draw[right hook->] (m-1-1) to (m-1-3);
      \draw[->] (m-1-3) to node[below right]{$\pr_2$} (m-2-2);
     \end{tikzpicture}
   \end{center}
   the fiber of $\Delta(\mathfrak{X}_{\Sigma})\to N^1(\mathfrak{X}_{\Sigma})_{\R}$ over any big and nef class $[\D]$ in $N^1(\mathfrak{X}_{\Sigma})_{\Q}$ is 
   \[
   \phi_{\R}\left(\Delta_{\D}\right) = \phi_{\R}\left(\Delta_{[\D]}\right) = \Delta_{A_{\D}},
   \]
 where $\phi_{\R}\colon M_{\R} \simeq \R^n$ is the identification  of (\ref{ident}), i.e.~we have that
   \[
   \pr_2^{-1}([\D])\cap\Delta(\mathfrak{X}_{\Sigma}) = \phi_{\R}(\Delta_{\D}) \subseteq \R^n \times \{[\D]\} = \R^n.
   \]
   The convex cone $\Delta(\mathfrak{X}_{\Sigma})$ is referred to as the \emph{global Okounkov $b$-body} of $\X$. 
\end{theorem} 
\begin{proof}
We define 
\[
\Delta(\mathfrak{X}_{\Sigma}) \coloneqq \varprojlim_{\Sigma'\in R(\Sigma)}\Delta(X_{\Sigma'}) \subseteq \R^n \times N^1(\mathfrak{X}_{\Sigma})_{\R},
\]
where $\Delta(X_{\Sigma'})$ is the global Okounkov body of the smooth and complete toric variety $X_{\Sigma'}$. This is a well defined closed set in the topological space $\R^n \times N^1(\mathfrak{X}_{\Sigma})_{\R}$. Indeed, the topology of the space $\R^n \times N^1(\mathfrak{X}_{\Sigma})_{\R}$ can be seen as a subspace topology of the product topology given on $\prod_{\Sigma' \in R(\Sigma)}\R^n \times N^1(X_{\Sigma'})_{\R}$. Hence, since all the sets $\Delta(X_{\Sigma'})$ are closed in $\R^n \times N^1(X_{\Sigma'})_{\R}$, their product and hence their inverse limit is a closed set in $\R^n \times N^1(\mathfrak{X}_{\Sigma})_{\R}$. Moreover, let $[\D] \in N^1(\mathfrak{X}_{\Sigma})_{\Q}$ be the class of a big and nef toric $b$-divisor. Since for all $\Sigma' \in R(\Sigma)$, the Okounkov bodies $\Delta_{A_{D_{\Sigma'}}} $ only depend on the class $[D_{\Sigma'}]$ of the toric divisor and since, by nefness, $\Delta_{A_{\D}}$ can be seen as the limit of the $\Delta_{A_{D_{\Sigma'}}}$ as $\Sigma'$ runs through $R(\Sigma)$, we conclude that $\Delta_{A_{\D}}$ also only depends on the class. Whence, we have 
\begin{align*}
\text{pr}_2^{-1}(\D)\cap\Delta\left(\mathfrak{X}_{\Sigma}\right) & =  \text{pr}_2^{-1}\left(\varprojlim_{\Sigma' \in R(\Sigma)}D_{\Sigma'}\right)\cap \varprojlim_{\Sigma'\in R(\Sigma)}\Delta\left(X_{\Sigma'}\right) \\ & = \varprojlim_{\Sigma'\in R(\Sigma)}\left(\text{pr}_2^{-1}\left(D_{\Sigma'}\right) \cap \Delta\left(X_{\Sigma'}\right)\right) \\& = \varprojlim_{\Sigma'\in R(\Sigma)}\Delta_{A_{D_{\Sigma'}}} \times [\D] \; \left(\subseteq \R^n \times\varprojlim_{\Sigma'\in R(\Sigma)}N^1(X_{\Sigma'})_{\R} \right)\\& = \Delta_{A_{\D}} = \phi_{\R}\left(\Delta_{\D}\right),
\end{align*} 
where the third equality follows from Theorem \ref{global_body} and the last one follows from Theorem~\ref{equal_conv}.
\end{proof}
\begin{rem}
By Remark \ref{big_cone} we can describe the image of $\Delta(\mathfrak{X}_{\Sigma})$ in $N^1(\mathfrak{X}_{\Sigma})_{\R}$ under the projection $\text{pr}_2$ as follows: for any $\Sigma' \in R(\Sigma)$, the incarnation in the vector space $N^1(X_{\Sigma'})_{\R}$ of the image of $\Delta(\mathfrak{X}_{\Sigma})$ in $N^1(\mathfrak{X}_{\Sigma})_{\R}$ is the pseudoeffective cone $\overline{\text{Eff}(X_{\Sigma'})}$.
\end{rem}
We denote by $ \BBig\left(\mathfrak{X}_{\Sigma}\right)$ and by $\Nef\left(\mathfrak{X}_{\Sigma}\right)$ the set of classes of big and nef toric $b$-divisors on $\X$, respectively. 
The following corollary is the analogue of \cite[Corollary~4.12]{KK} in the $b$-context. 
\begin{cor}
There is a uniquely defined continuous function 
\[
\deg_{\mathfrak{X}_{\Sigma}} \colon \BBig\left(\mathfrak{X}_{\Sigma}\right) \cap \Nef\left(\mathfrak{X}_{\Sigma}\right) \longrightarrow \R
\]
that computes the degree of any big and nef toric $b$-divisor class. This function is homogeneous of degree $n$ and log-concave, i.e 
\[
\deg_{\mathfrak{X}_{\Sigma}}(\D+\pmb{E})^{1/n} \geq \deg_{\mathfrak{X}_{\Sigma}}(\D)^{1/n} + \deg_{\mathfrak{X}_{\Sigma}}(\pmb{E})^{1/n}
\]
for any $\D, \pmb{E}$ in $\BBig\left(\mathfrak{X}_{\Sigma}\right) \cap \Nef\left(\mathfrak{X}_{\Sigma}\right)$. 
\end{cor}
\begin{proof}
We take 
\[
\deg_{\mathfrak{X}_{\Sigma}} = \D^n = n!\,\text{vol}\left(\Delta_{\D}\right).
\]
Note that  $\vol\left(\Delta_{\D}\right) = \vol\left(\Delta_{A_{\D}}\right) = \vol\left(\Delta\left(\mathfrak{X}_{\Sigma}\right)_{\D}\right)$ corresponds to the volume of the fibre of the projection $\text{pr}_2$. Then, log-concavity follows from the Brunn--Minkowski inequality in Corollary~\ref{brunn_mink}. And since concave functions are continuous in the interior of their domain, the statement of the corollary follows. 
\end{proof}

\printbibliography
\vspace{2cm}
\noindent Ana María Botero\\
Institut für Mathematik\\
Technische Universität Darmstadt\\
Karolinenplatz 5\\
64289 Darmstadt\\
Germany\\
e-mail: \url{botero@mathematik.tu-darmstadt.de}
\end{document}